\pdfoutput=1
\RequirePackage{ifpdf}
\ifpdf 
\documentclass[pdftex]{sigma}
\else
\documentclass{sigma}
\fi

\numberwithin{equation}{section}

\newtheorem{Theorem}{Theorem}[section]
\newtheorem{Lemma}[Theorem]{Lemma}
\newtheorem{Proposition}[Theorem]{Proposition}
\newtheorem*{Goal}{Goal}
 { \theoremstyle{definition}
\newtheorem{Definition}[Theorem]{Definition}
\newtheorem{Example}[Theorem]{Example}
\newtheorem{Remark}[Theorem]{Remark} }

\newcounter{pphcounter}[subsection]
\renewcommand{\thepphcounter}{\thesubsection.\arabic{pphcounter}}
\newcommand{\pph}[1]{\noindent
\refstepcounter{pphcounter}\bf\thepphcounter. #1\rm}

\newcommand{\aA}{\mathbb A}
\newcommand{\RR}{\mathbb R}
\newcommand{\CC}{\mathbb C}
\newcommand{\QQ}{\mathbb Q}
\newcommand{\ZZ}{\mathbb Z}
\newcommand{\PP}{\mathbb P}
\newcommand{\NN}{\mathbb N}
\newcommand{\FF}{\mathbb F}
\newcommand{\MM}{\mathbb M}
\newcommand{\TT}{\mathbb T}
\newcommand{\II}{\mathbb I}
\newcommand{\cA}{\mathcal A}
\newcommand{\cB}{\mathcal B}

\newcommand{\cD}{\mathcal D}

\newcommand{\cF}{\mathcal F}
\newcommand{\SF}{[\mathcal F]}
\newcommand{\SFL}{\SF_{\Lambda}}
\newcommand{\cE}{\mathcal E}
\newcommand{\cM}{\mathcal M}
\newcommand{\cS}{\mathcal S}
\newcommand{\cSb}{\cS_\bullet}
\newcommand{\cSw}{\cS_\circ}
\newcommand{\cW}{\mathcal W}
\newcommand{\ocE}{{\cE_{\gL}} }
\newcommand{\ocA}{\cA_\gL}
\newcommand{\ocM}{\cM_\gL}
\newcommand{\ocW}{\cW_\gL}
\newcommand{\oocW}{\overline{\cW}_\gL}
\newcommand{\WLD}{\ocW^\vee}
\newcommand{\oWLD}{\oocW^\vee}
\newcommand{\vge}{\varepsilon}
\newcommand{\gep}{\vge_+}
\newcommand{\gem}{\vge_-}
\newcommand{\geb}{\vge_\bullet}
\newcommand{\gew}{\vge_\circ}
\newcommand{\inp}{\raisebox{.2ex}{$\centerdot$}}
\newcommand{\sD}{\mathsf D}
\newcommand{\sDf}{\sD^\flat}
\newcommand{\sR}{\mathsf R}
\newcommand{\sRf}{\sR^\flat}
\newcommand{\sP}{\mathsf P}
\newcommand{\spv}{{\sP}^\star}
\newcommand{\spb}{{\sP}^\bullet}
\newcommand{\spw}{{\sP}^\circ}
\newcommand{\osP}{\sP_{\!\gL}}
\newcommand{\ospv}{\osP^\star}
\newcommand{\ospb}{\osP^\bullet}
\newcommand{\ospw}{\osP^\circ}
\newcommand{\ospz}{\osP^{\mathsf z}}
\newcommand{\gL}{\Lambda}
\newcommand{\ga}{\alpha}
\newcommand{\cga}{\check{\ga}}
\newcommand{\gb}{\beta}
\newcommand{\cgb}{\check{\gb}}
\newcommand{\gs}{\sigma}
\newcommand{\vgs}{\varsigma}
\newcommand{\gl}{\lambda}
\newcommand{\ogs}{\gs}
\newcommand{\fv}{\mathfrak v}
\newcommand{\fp}{\mathfrak{p}}
\newcommand{\ma}{\mathfrak{A}}
\newcommand{\mav}{{\ma}^{\star\star}}
\newcommand{\bv}{\mathbf v}
\newcommand{\bx}{\mathbf x}
\newcommand{\bt}{\mathbf t}
\newcommand{\bs}{\mathbf s}
\newcommand{\bb}{\mathbf b}
\newcommand{\bw}{\mathbf w}
\newcommand{\bz}{\mathbf z}
\newcommand{\bp}{\mathbf p}
\newcommand{\ubp}{[\bp]}
\newcommand{\bi}{\mathbf i}
\newcommand{\bh}{\mathbf h}
\newcommand{\sfq}{\mathsf q}
\newcommand{\sm}{\mathsf m}
\newcommand{\sn}{\mathsf n}
\newcommand{\ve}{\mathsf {vec}}
\newcommand{\nul}{\mathsf {0}}
\newcommand{\sZ}{\mathsf {Z}}
\newcommand{\sT}{\mathsf {T}}
\newcommand{\auto}{\mathsf {Aut}}
\newcommand{\rank}{\operatorname{\mathsf{rank}}}
\newcommand{\conv}{\mathsf{conv}}
\newcommand{\Jac}{\mathsf{Jac}}
\newcommand{\diag}{\mathsf{diag}}
\newcommand{\perm}{\mathsf{Perm}}
\newcommand{\pad}{\mathsf{Path}}
\newcommand{\matel}{{\mathsf{Mat}}_{\ospv}}
\newcommand{\Spec}{\mathsf{Spec}}
\newcommand{\Proj}{\mathsf{Proj}}
\newcommand{\V}{\omega}
\newcommand{\EV}{\Omega}
\newcommand{\sign}{\mathrm{sign}}
\newcommand{\ho}{\mathrm{H}}
\newcommand{\Hom}{\mathrm{Hom}}
\newcommand{\map}{\mathrm{Maps}}
\newcommand{\gG}{\Gamma}
\newcommand{\gGL}{\gG_{\!\!\gL}}
\newcommand{\gGLD}{\gG_{\!\!\gL}^\vee}
\newcommand{\ZZF}{\textsc{czzf}}
\newcommand{\ze}[1]{\mathsf{Z}^{#1}}
\newcommand{\ol}[1]{\overline{#1}}
\newcommand{\breuk}[2]{\textstyle{\frac{#1}{#2}}\displaystyle}

\begin{document}
\allowdisplaybreaks

\newcommand{\arXivNumber}{1805.09627}

\renewcommand{\thefootnote}{}

\renewcommand{\PaperNumber}{110}

\FirstPageHeading

\ShortArticleName{Zhegalkin Zebra Motives Digital Recordings of Mirror Symmetry}

\ArticleName{Zhegalkin Zebra Motives\\ Digital Recordings of Mirror Symmetry\footnote{This paper is a~contribution to the Special Issue on Modular Forms and String Theory in honor of Noriko Yui. The full collection is available at \href{http://www.emis.de/journals/SIGMA/modular-forms.html}{http://www.emis.de/journals/SIGMA/modular-forms.html}}}

\Author{Jan STIENSTRA}

\AuthorNameForHeading{J.~Stienstra}

\Address{Department of Mathematics, Utrecht University,\\ P.O.~Box 80010, 3508 TA Utrecht, The Netherlands}

\Email{\href{mailto:J.Stienstra@uu.nl}{J.Stienstra@uu.nl}}

\ArticleDates{Received May 10, 2018, in final form October 02, 2018; Published online October 13, 2018}

\Abstract{Zhegalkin zebra motives are tilings of the plane by black and white polygons representing certain ${\mathbb F}_2$-valued functions on ${\mathbb R}^2$. They exhibit a rich geometric structure and provide easy to draw insightful visualizations of many topics in the physics and mathematics literature. The present paper gives some pieces of a general theory and a few explicit examples.
Many more examples will be shown in the forthcoming article ``Zhegalkin zebra motives: algebra and geometry in black and white''.}

\Classification{52C20; 82B20; 14M25}

\tableofcontents

\renewcommand{\thefootnote}{\arabic{footnote}}
\setcounter{footnote}{0}

\section{Introduction}\label{sec:intro}

\subsection{Zhegalkin zebra functions}\label{subsec:intro ZZF}

The constructions of motives in algebraic geometry heavily depend on the intersection theory of algebraic cycles and, hence, on the fairly delicate choice of an adequate equivalence relation on algebraic cycles. \textit{Chow motives}, for instance, are based on \textit{rational equivalence}, which is the finest equivalence relation on algebraic cycles yielding a good intersection theory \cite{MNP}.

On the contrary, the ``motives'' in the present paper are built with the usual set theoretical operations from simple subsets of the plane, which we call zebras. In 1927 Zhegalkin pointed out that functions with values in the field $\FF_2=\ZZ/2\ZZ$ with the usual addition and multiplication can replace the standard Boolean formalism. The \textit{zebra with frequency} $\fv\in\RR^2$, $\fv\neq\nul$, is the function on~$\RR^2$ given by
\begin{gather}\label{eq:zebra}
\ze{\fv}(\bx)=\lfloor 2\bx\inp\fv\rfloor\bmod 2\qquad\textrm{for}\quad \bx\in\RR^2.
\end{gather}
Here $\inp$ is the dot product on $\RR^2$ and for a real number $r$ the integer $\lfloor r\rfloor$ is such that $0\leq r-\lfloor r\rfloor<1$. It is sometimes convenient to identify the Euclidean plane $\RR^2$ and the complex plane $\CC$. In this paper we only use zebras for which the frequencies are positive integer multiples of the complex numbers
\begin{gather}\label{eq:frequencies}
\fv_1=\sqrt{3} \vge^5,\qquad\fv_2=\vge^4,\qquad\fv_3=\sqrt{3} \bi ,\qquad\fv_4=\vge^2,\qquad \fv_5=\sqrt{3} \vge,\qquad\fv_6=1 ,
\end{gather}
with $\bi=\sqrt{-1}$ and $\vge=e^{\pi \bi/6}$; see Fig.~\ref{fig:zebra-frequencies}.

\begin{figure}[h!]\centering
\includegraphics{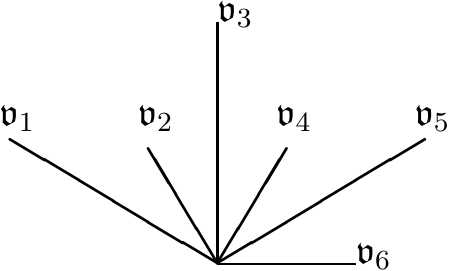}\qquad
\includegraphics[width=30mm]{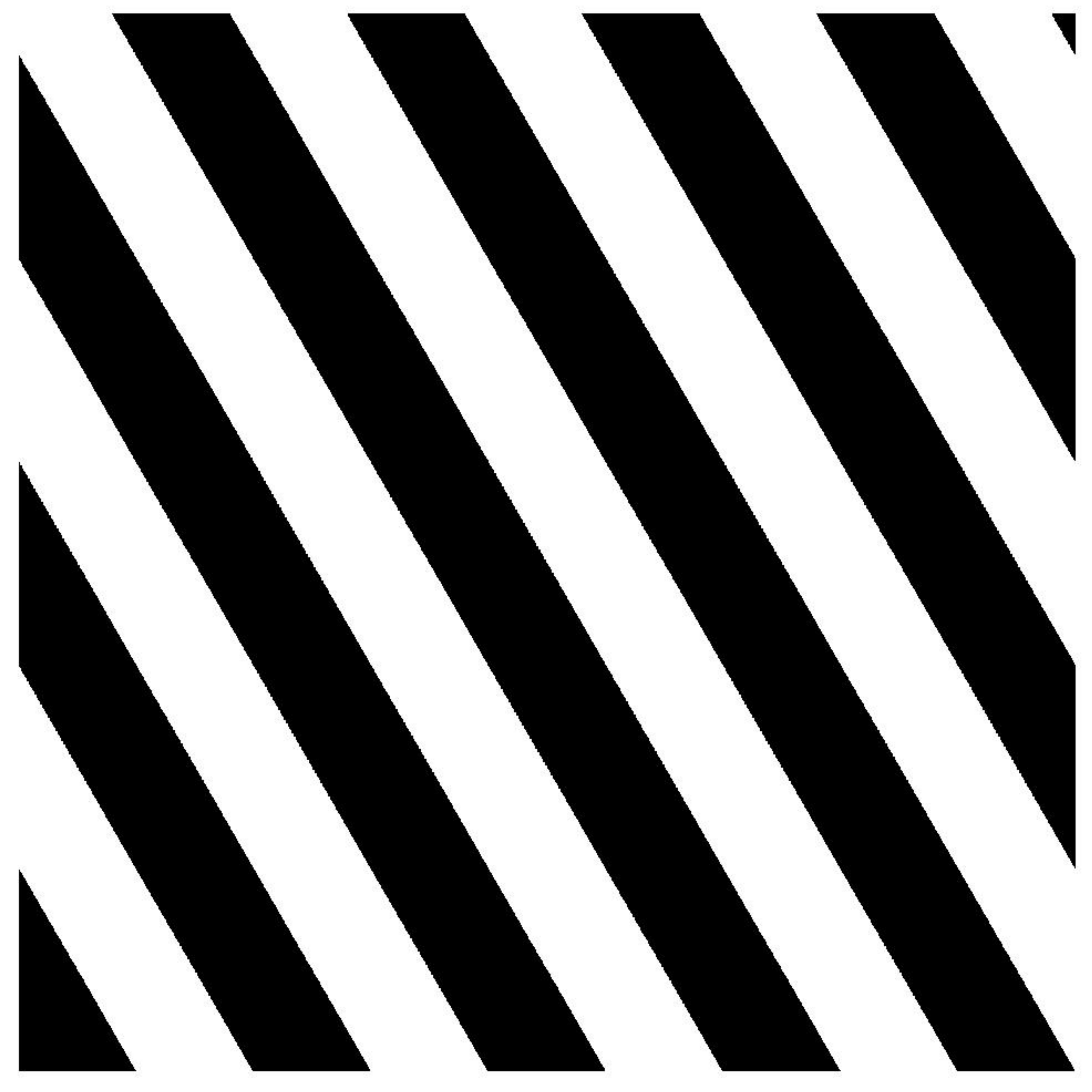}

\caption{Left: \eqref{eq:frequencies}. Right: Zebra $\ze{\fv}$. All bands are $\perp\fv$ and have width $\frac{1}{2|\fv|}$.} \label{fig:zebra-frequencies}
\end{figure}

\begin{Definition}\label{def:zzf}We denote the zebra with frequency $k\fv_j$ by $\ze{jk}$. The elements of the ring of $\FF_2$-valued functions on $\RR^2$ generated by the zebras $\ze{jk}$ are called \textit{Zhegalkin zebra functions}. Such a function $\cF$ gives a tiling of the plane by white ($\cF=0$) and black ($\cF=1$) polygons. A~Zhegalkin zebra function is \emph{convex} (\ZZF) if all polygons in the tiling are bounded and convex. These tilings are the \textit{Zhegalkin zebra motives} in the title.
\end{Definition}

\begin{Goal} Understand the deformation theory of these tilings.
\end{Goal}

Every Zhegalkin zebra function can be written as a \textit{Zhegalkin zebra polynomial}, i.e., a polynomial in the variables $\ze{jk}$ in which all monomials have coefficients $1$ and the variables in each monomial have exponent $1$. Section~\ref{subsec:draw} describes an efficient way for evaluating a Zhegalkin zebra polynomial and drawing the black-white picture. Convexity can easily be checked by visual inspection. See Figs.~\ref{fig:classic1}, \ref{fig:classic2}, \ref{fig:kagome} and~\ref{fig:model12b1} for examples.

A Zhegalkin zebra function $\cF$ has an automorphism group consisting of translations leaving the tiling invariant:
\begin{gather*}
\auto(\cF) = \big\{ \tau\in\RR^2\,|\,\cF(\bx+\tau)=\cF(\bx) ,\,\forall\, \bx\in\RR^2 \big\} .
\end{gather*}
This is a lattice in $\RR^2$ if $\cF$ is convex. For every sublattice $\gL\subset\auto(\cF)$ the function $\cF$ descends to a function on the torus~$\RR^2/\gL$ and gives a tiling of this torus by black and white polygons.

This brings us to the setting of dimer models (a.k.a.\ brane tilings), quivers with superpotential and discrete differential geometry. There is an extensive literature on these topics written from very different view-points, with very different terminologies, for very different applications. Our view-point will be that the pictures are \textit{realizations} of an underlying combinatorial structure. Our goal is:
\textit{Understand the deformations of these realizations}.

\subsection{The superpotential and weight functions}\label{subsec:intro superpotential}
The combinatorial structure consists of the set $\cE$ of edges in the picture, two permuta\-tions~$\gs_0$,~$\gs_1$ of~$\cE$ and an injective homomorphism $\fp\colon \ZZ^2\rightarrow\perm(\cE)$ into the permutation group of~$\cE$. The cycles (=~orbits) of $\gs_0$ and $\gs_1$ correspond to the \textit{oriented boundaries} of the white and black polygons, while the cycles of the permutation $\gs_2=\gs_1^{-1}\gs_0$ correspond to the \textit{vertices} in the tiling. The orientation of the edges is such that the boundaries of the black (resp.\ white) polygons are oriented clockwise (resp. counter-clockwise). The homomorphism~$\fp$ comprises the action of $\auto(\cF)$ and an isomorphism $\auto(\cF)\simeq\ZZ^2$. The permutations $\gs_0$ and $\gs_1$ commute with this action. Associated with a sublattice $\gL$ of $\ZZ^2$ is then the \textit{finite set} $\ocE=\cE/\gL$ equipped with the permutations $\gs_0$, $\gs_1$, $\gs_2$ and an action of the finite group $\ZZ^2/\gL$. Since the torus $\RR^2/\gL$ has genus $1$ the numbers of cycles of the permutations satisfy $|\gs_0|+|\gs_1|+|\gs_2|=|\ocE|$.
\begin{Definition}\label{def:superpotential} We call $\SFL=(\ocE,\gs_0,\gs_1)$ the \textit{superpotential} of the Zhegalkin zebra function~$\cF$ and the lattice $\gL$.
\end{Definition}

\begin{Remark}\label{rem:physics potential}\looseness=-1 In the physics literature, e.g.,~\cite{HS}, one writes the superpotential as a sum of $|\gs_0|+|\gs_1|$ monomials in \textit{non-commuting variables} $\big\{\widetilde{X}_e\big\}_{e\in\ocE}$ such that a cycle $(e_1,\ldots,e_r)$ of~$\gs_c$, $c=0,1$, contributes the monomial $(-1)^c \widetilde{X}_{e_1}\cdots\widetilde{X}_{e_r}$. In this formulation the cyclic structure is implicit.
\end{Remark}

\begin{figure}[t]

\centerline{\includegraphics[width=6.6cm]{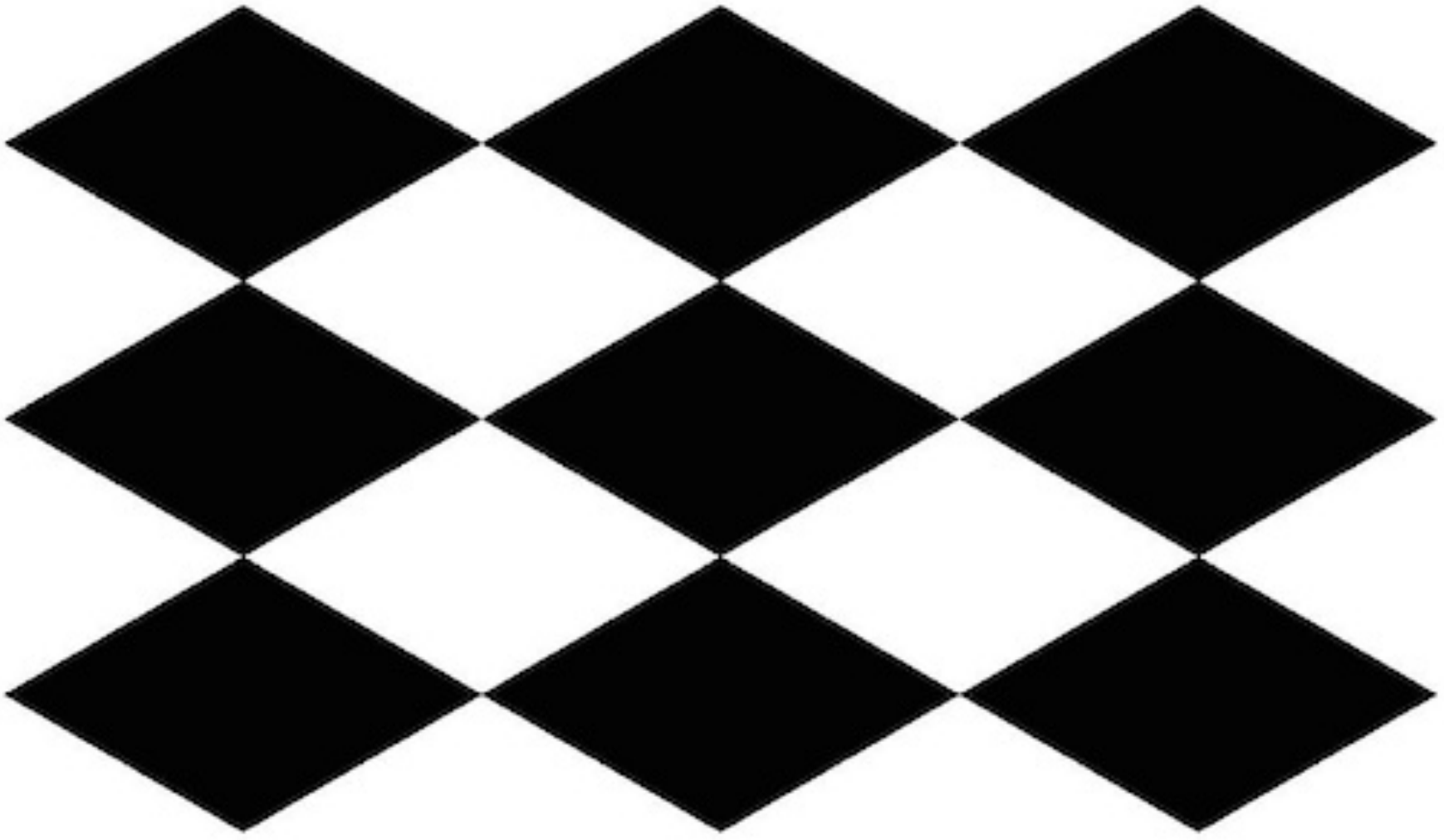} \qquad\quad \includegraphics[width=6.6cm]{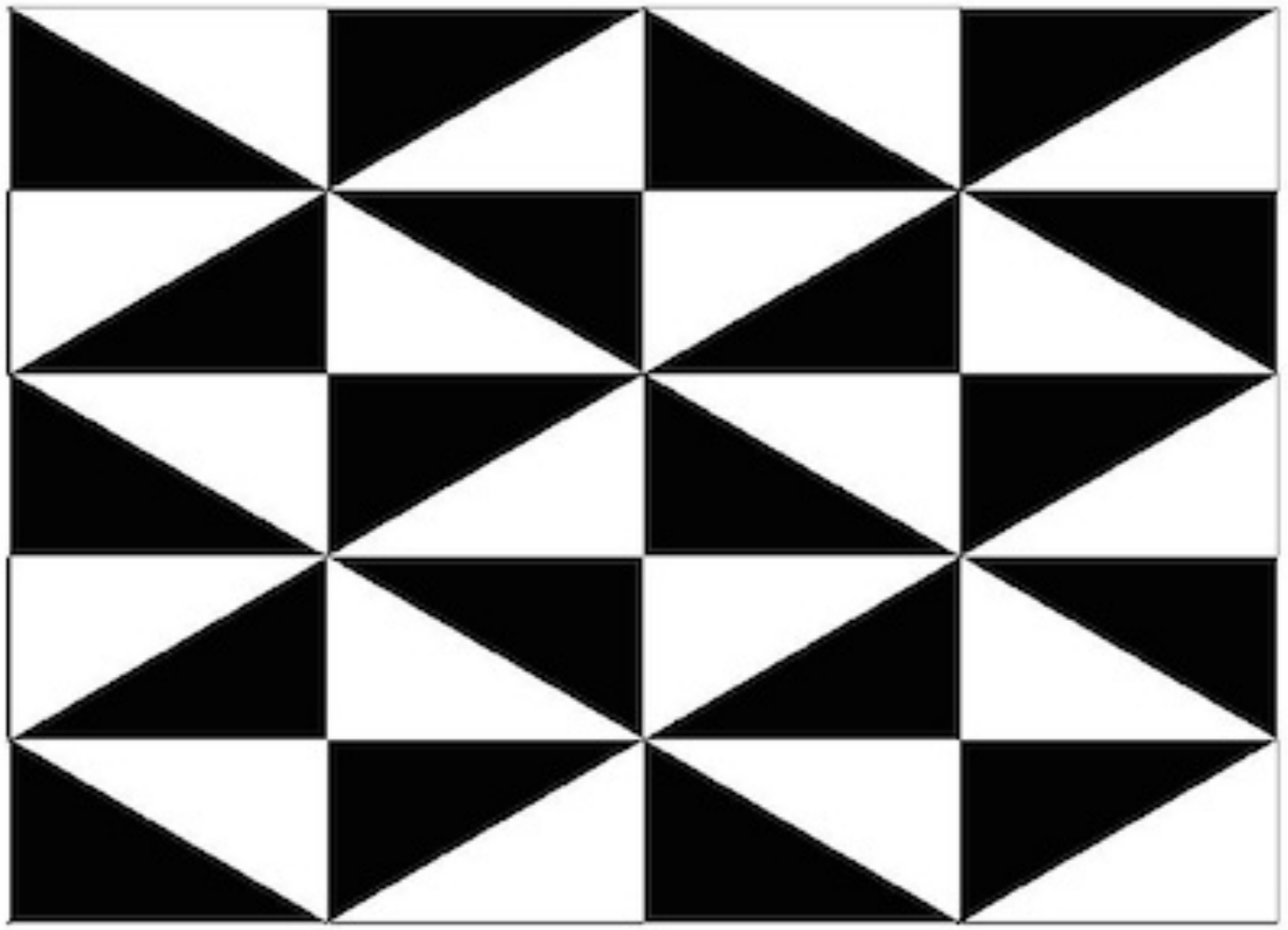}}

\hspace*{20mm} $\begin{array}{l}
\cF_2=\ze{21}+\ze{41}\\
\auto(\cF_2)= \ZZ\fv_6\oplus\frac{1}{3}\ZZ\fv_3\\
\gs_0=(1,2,3,4)\\
\gs_1=(1,4,3,2)\end{array}$ \hspace{22mm}
$\begin{array}{l}
\cF_4=\ze{21}+\ze{31}+\ze{41}+\ze{61}\\
\auto(\cF_4)= \ZZ\fv_6\oplus\frac{1}{3}\ZZ\fv_3\\
\gs_0=(1,2,3)(4,5,6)(7,8,9)(10,11,12)\\
\gs_1=(1,11,9)(4,8,12)(7,2,6)(10,5,3)
\end{array}$
\caption{Some classical patterns realized as \ZZF, with their automorphism groups and corresponding superpotentials.}\label{fig:classic1}
\vspace{-2mm}
\end{figure}

\begin{figure}[t]

\hspace*{3mm}\includegraphics[width=5cm]{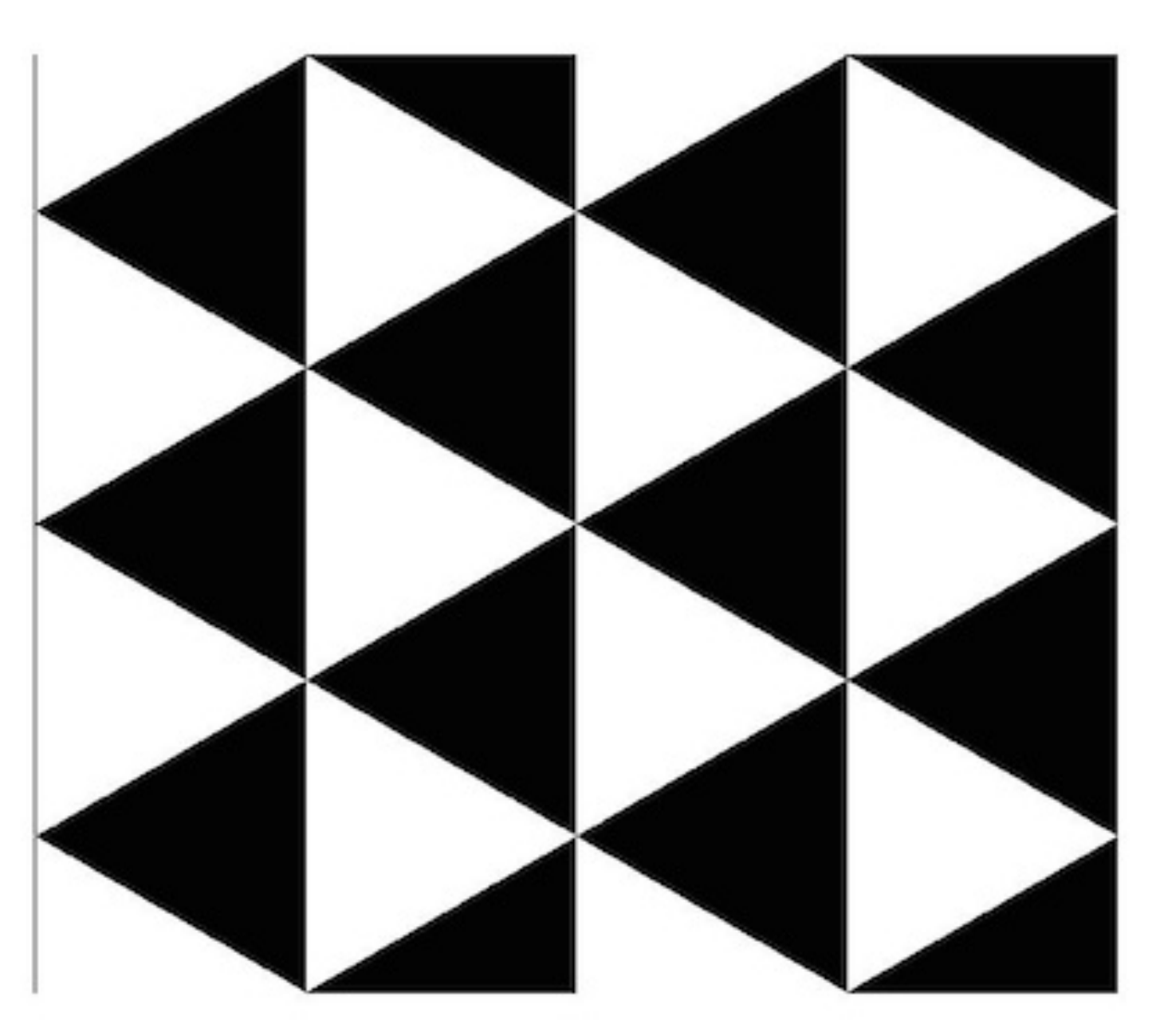}\hspace{17mm} \includegraphics[width=7.5cm]{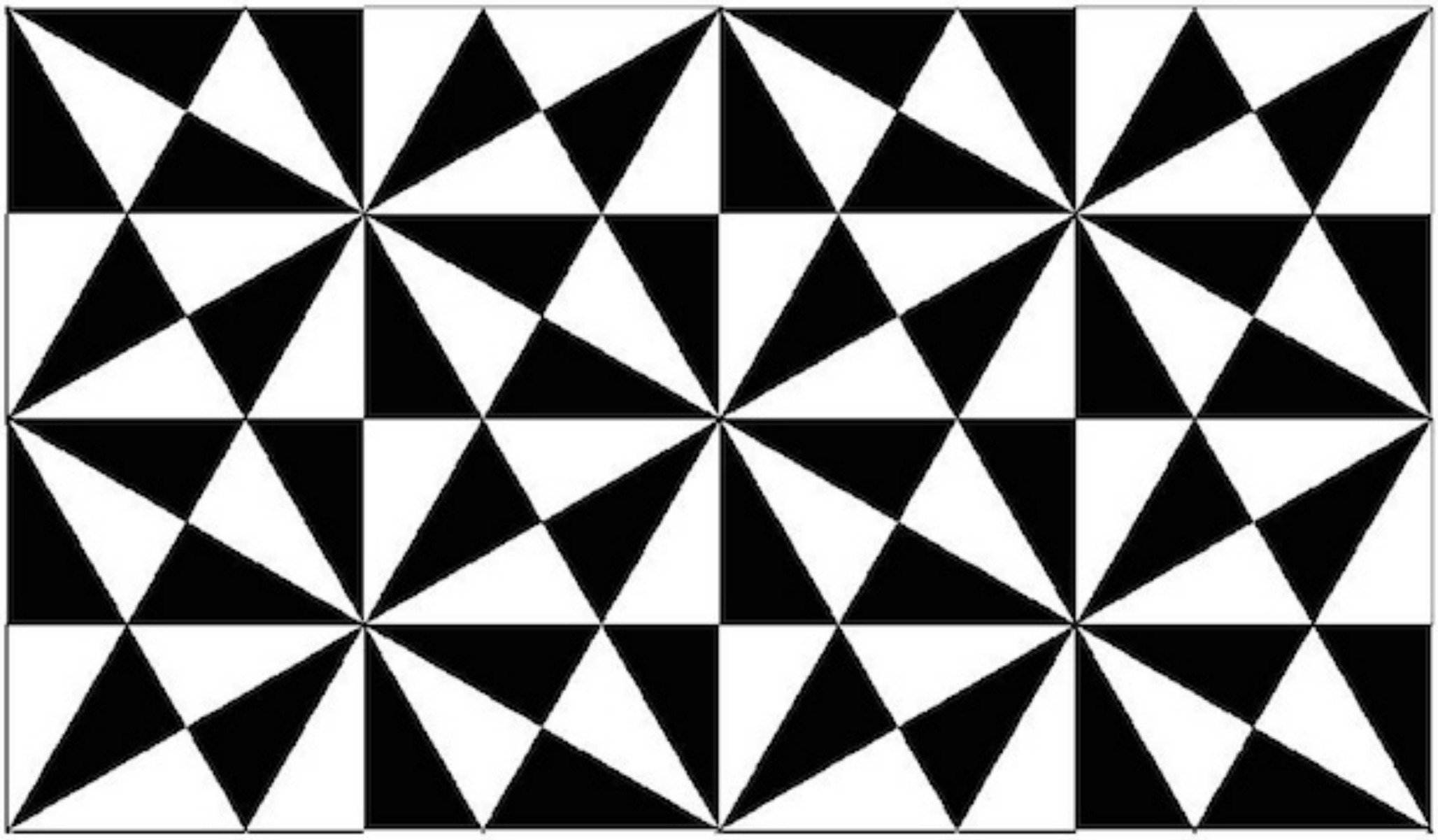}

\hspace*{8mm}$\begin{array}{l}
\cF_3=\ze{21}+\ze{41}+\ze{61}\\
\auto(\cF_3)=\ZZ\fv_6\oplus\frac{1}{3}\ZZ\fv_5\\
\gs_0=(1,2,3)\\
\gs_1=(1,3,2)
\end{array}$ \hspace{6mm}
$\begin{array}{l}
\cF_6=\ze{11}+\ze{21}+\ze{31}+\ze{41}+\ze{51}+\ze{61}\\
\auto(\cF_6)=\frac{1}{3}\ZZ\fv_1\oplus\frac{1}{3}\ZZ\fv_3\\
\gs_0=(1,2,3)(4,5,6)(7,8,9)(10,11,12)(13,14,15)(16,17,18)\\
\gs_1=(1,8,18)(4,11,3)(7,14,6)(10,17,9)(13,2,12)(16,5,15)
\end{array}$

\caption{Some classical patterns realized as \ZZF, with their automorphism groups and corresponding superpotentials.}\label{fig:classic2}
\end{figure}

The superpotential $\SFL=(\ocE,\gs_0,\gs_1)$ provides three abstract graphs
\begin{gather*}
\gGL = \big(s,t\colon \ocE \rightrightarrows\ospv\big) ,\qquad\gGLD = \big(b,w\colon \ocE\rightrightarrows \ospb\cup\ospw\big) ,\\
\cD_{\cF,\gL} = \textsf{vertices}\colon \ \ospv\cup\ospb\cup\ospw,\\
\hphantom{\cD_{\cF,\gL} =}{} \ \textsf{edges}\colon \
\{(\bw,\bv)\in\ospw\times\ospv\,|\,\exists\, e\in\ocE\colon \bw=w(e),\,\bv=t(e)\textrm{ or } s(e)\}\\
\hphantom{\cD_{\cF,\gL} = \textsf{edges}\colon} \ \cup\{(\bb,\bv)\in\ospb\times\ospv\,|\,\exists\, e\in\ocE\colon \bb=b(e),\,\bv=t(e)\textrm{ or } s(e)\}.
\end{gather*}
Here $\ospv$, $\ospb$, $\ospw$ denote the respective sets of vertices, black and white polygons in the tiling of the torus $\RR^2/\gL$ and $s$, $t$, $b$, $w$ are the respective maps which assign to an edge its source, target, adjacent black and white polygons.

In addition to the superpotential the actual pictures also contain a map $\V\colon \cE\rightarrow\RR^2\setminus\{\nul\}$ which specifies for every edge the corresponding vector in $\RR^2$. For a sublattice $\gL\subset\ZZ^2$ we want this specification to be $\gL$-invariant; i.e., it should be a map $\V\colon \ocE\rightarrow\RR^2\setminus\{\nul\}$. We call such a~map $\V$ a~\textit{realization} of $\SFL$. It also gives a realization $\gL_\V$ of the lattice $\gL$ by translations in the plane which leave the tiling specified by $\V$ invariant. We denote the corresponding torus by
\begin{gather*}
\TT_\V = \RR^2/\gL_\V .
\end{gather*}
The Zhegalkin zebra function $\cF$ provides a tiling of $\RR^2$ with automorphism group $\auto(\cF)$ and hence a realization $\V_\cF$ of $\SFL$ for every sublattice $\gL$ of $\ZZ^2$. It identifies $\gL$ with a sublattice $\gL_{\V_\cF}$ of~$\auto(\cF)$.

A realization $\V$ of the superpotential gives a tiling of~$\RR^2$ which modulo $\gL_\V$ gives an embedding of the quiver (=~graph with oriented edges) $\gGL$ into the torus $\TT_\V$ as the 0-cells and 1-cells in the tiling. One can subsequently embed the graphs $\gGLD$ and $\cD_{\cF,\gL}$ into this torus by means of a function $\theta\colon \ocE\rightarrow\RR_{>0}$ for which the sum over each cycle of $\gs_0$ and each cycle of $\gs_1$ is equal to $1$. This function is used to mark in each black/white polygon a point by taking a convex combination of the midpoints of its edges. This will be discussed in detail in Section~\ref{sec:draw GD}. In~\cite{GK} such a~function $\theta$ is called a (\textit{positive}) \textit{fractional matching}. The existence of a fractional matching for $\SFL$ implies $|\ospw|=|\ospb|$.

\begin{Definition}\label{def:weight}An \textit{integer weight function for the superpotential $\SFL$} is a map $\nu\colon \ocE\rightarrow\ZZ_{\geq0}$ for which the sum over each cycle of $\gs_0$ and each cycle of $\gs_1$ is equal to an integer $\deg \nu$ (the \textit{degree of~$\nu$}). The integer weight functions with the operation $+$ form a graded semi-group~$\ocW$. An integer weight function of degree $1$ is called a \textit{perfect matching}, \textit{dimer covering} or \textit{dimer configuration}. The set of perfect matchings is denoted by~$\ocM$.

An integer weight function $\nu$ is said to be \textit{positive} if $\nu(e)>0$ for all $e\in\ocE$.
\end{Definition}

Perfect matchings play a crucial role all over the literature on \textit{dimer models}.
From the permutations $\gs_0$ and $\gs_1$ one can easily check whether perfect matchings exist and determine them all.
Subsequently one can check whether the sum of all perfect matchings is a positive weight function, which then divided by its degree $|\ocM|$ yields a positive fractional matching.

\begin{Definition}\label{def:complete}We say that the superpotential $\SFL$ is \textit{dimer complete} if the sum of all perfect matchings is a positive weight function.
\end{Definition}

If $\SFL$ is dimer complete, the semi-group $\ocW$ is generated by the perfect matchings (see Proposition \ref{prop:matching polytope}):
\begin{gather}\label{eq:weight span I}
\ocW = \ZZ_{\geq0}\ocM .
\end{gather}

Rescaling the axes in the picture of the realization $\V_\cF$ by means of the diagonal matrix $\diag\bigl(\frac{1}{2},\frac{\sqrt{3}}{2}\bigr)$ yields a realization of the superpotential $\SFL$ with edge vectors $\V_\cF(e) \diag\bigl(\frac{1}{2},\frac{\sqrt{3}}{2}\bigr)$. It follows from~\eqref{eq:zebra} (see also Section~\ref{step 1}) that the vertices in the rescaled tiling have coordinates in $\QQ$. By further rescaling with a positive integer factor one can clear the denominators and obtain a realization $\V$ of $\SFL$ with edge vectors in $\ZZ^2$, say $\V(e)=(\V_1(e),\V_2(e))$ with $\V_1,\V_2\colon \ocE\rightarrow \ZZ$.

Now assume that the superpotential $\SFL$ is dimer complete and let $\nu$ be a positive integer weight function. Then, for a sufficiently large integer~$N$ the maps $\nu_1=\V_1+N\nu$, $\nu_2=\V_2+N\nu$ and $\nu_3=N\nu$ are positive weight functions with $\deg \nu_1=\deg \nu_2=\deg \nu_3$. Then $\V=(\nu_1-\nu_3, \nu_2-\nu_3)$ and $\theta=\frac{1}{\deg \nu_3}\nu_3$ yield for every edge $\widetilde{e}$ in the tiling of $\RR^2$ four points $s(\widetilde{e})$, $t(\widetilde{e})$, $b(\widetilde{e})$, $w(\widetilde{e})$, namely the endpoints of that edge and the marked (by $\theta$) points in the polygons adjacent to that edge; see Figs.~\ref{fig:see quadrangles}, \ref{fig:quadrangle}, \ref{fig:kagome} and \ref{fig:quadrangle3}.

These quadrangles (for $e\in\ocE$) constitute a tiling of $\RR^2$. Taken modulo $\gL_\V$ the vertices and edges of the induced quadrangle-tiling give an embedding of the graph $\cD_{\cF,\gL}$ into the torus $\TT_\V$. So $\cD_{\cF,\gL}$ is an \textit{S-quad-graph} in the sense of \cite[Definitions~3.1 and 4.3]{Bob}.

\begin{Definition}\label{def:weight realization}Assume that the superpotential $\SFL$ is dimer complete. A \emph{weight realization} of $\SFL$ is a triple of positive weight functions $(\nu_1,\nu_2,\nu_3)$ such that $\V=(\nu_1-\nu_3, \nu_2-\nu_3)$ is a~realization of $\SFL$ in which all black and white polygons and all quadrangles determined by $\big(\V,\frac{1}{\deg \nu_3}\nu_3\big)$ are strictly convex.
\end{Definition}

Since it is easy to draw pictures (see Section~\ref{subsec:comlatquad}) the conditions in Definition~\ref{def:weight realization} can easily be checked by visual inspection. In many examples one can find weight realizations by staring at the picture of the tiling for the Zhegalkin zebra function~$\cF$ drawn with the method of Section~\ref{subsec:draw}.

One can collect the maps $s,t\colon \ocE\rightarrow\ospv$ and $\nu_1,\nu_2,\nu_3\colon \ocE\rightarrow\ZZ_{>0}$ into a matrix $\mav\big(u_1^{\nu_1}u_2^{\nu_2}u_3^{\nu_3}\big)$ as follows. The rows and columns of $\mav\big(u_1^{\nu_1}u_2^{\nu_2}u_3^{\nu_3}\big)$ correspond with the elements of $\ospv$ and its entries lie in the polynomial ring $\ZZ[u_1,u_2,u_3]$; the entry in row $\bs$ and column $\bt$ is
\begin{gather}\label{eq:Anu123}
\mav\big(u_1^{\nu_1}u_2^{\nu_2}u_3^{\nu_3}\big)_{\bs,\bt} = \sum_{e\in\ocE: s(e)=\bs, t(e)=\bt} u_1^{\nu_1(e)}u_2^{\nu_2(e)}u_3^{\nu_3(e)} .
\end{gather}
The matrix $\mav\big(u_1^{\nu_1}u_2^{\nu_2}u_3^{\nu_3}\big)$ can be written uniquely as a sum
\begin{gather*}
	\mav\big(u_1^{\nu_1}u_2^{\nu_2}u_3^{\nu_3}\big) = \sum_{e\in\ocE}\Phi_{\nu_1,\nu_2,\nu_3}(e)
\end{gather*}
of matrices each of which has only one non-zero entry and this entry is a monomial; see Section~\ref{sec:jacmasterweight}. The algebra generated by the matrices $\Phi_{\nu_1,\nu_2,\nu_3}(e)$ is (isomorphic to) the \textit{Jacobi algebra} $\Jac(\SFL)$; see Theorem~\ref{thm:faithful}.

\begin{figure}[t]\centering
\includegraphics[width=4.5cm]{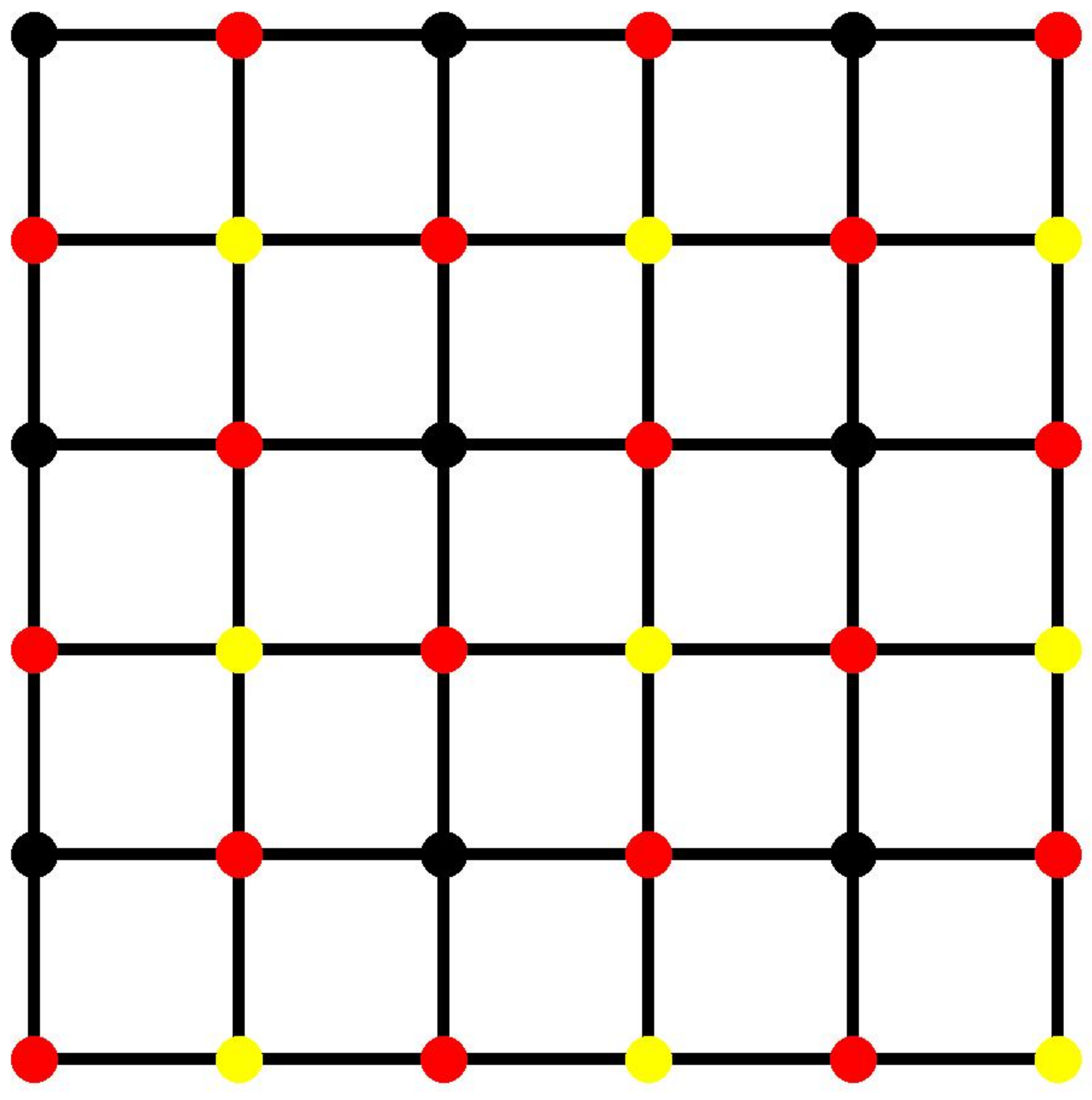}\qquad\qquad \includegraphics[width=4.5cm]{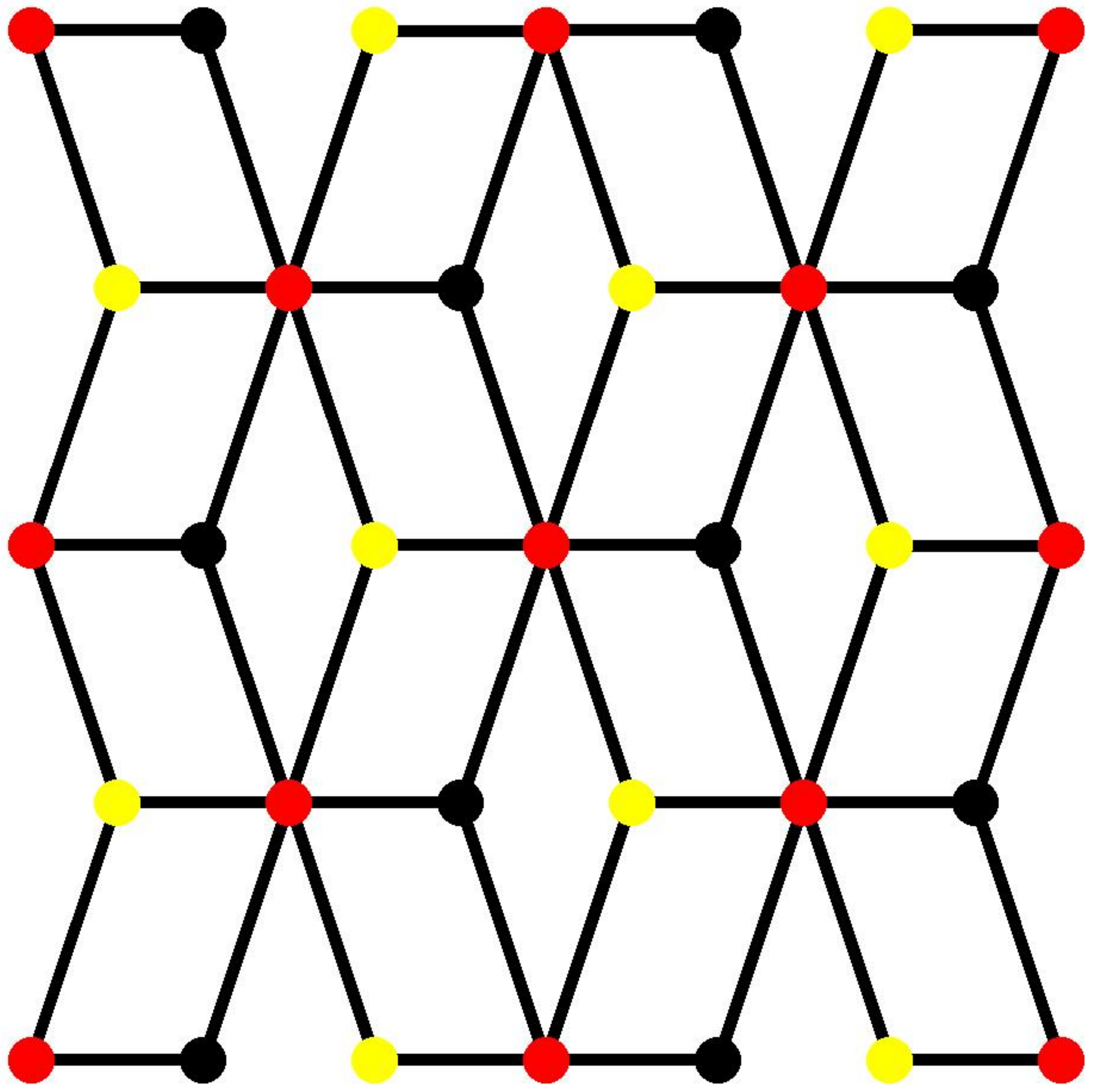}

\caption{Quadrangles for the matrices $\mav\big(u_1^{\nu_1}u_2^{\nu_2}u_3^{\nu_3}\big)$ in Example~\ref{exa:see weights}.}\label{fig:see quadrangles}
\end{figure}

\begin{Example}\label{exa:see weights} For $\cF_2$ and $\cF_3$ as in Figs.~\ref{fig:classic1} and \ref{fig:classic2} one has the weight realizations
\begin{gather*}
\left(\begin{matrix}
0&u_1^3u_2u_3^2+u_1u_2^3u_3^2\\
u_1u_2u_3^2+u_1^3u_2^3u_3^2&0
\end{matrix}\right)
\qquad \text{resp.}\quad u_1^3u_2u_3^3+u_1^4u_2^4u_3^3+u_1^2u_2^4u_3^3
\end{gather*}
The $2$-cells in these realizations are squares, resp.\ triangles with angles $\frac{\pi}{4}$, $\frac{\pi}{4}$, $\frac{\pi}{2}$. The marked points in the $2$-cells are their barycenters. The quadrangles for these weight realizations are shown in Fig.~\ref{fig:see quadrangles}.
\end{Example}

\subsection{Dessins d'enfants}\label{subsec:intro dessins}

The quadrangles in a weight realization constitute a tiling of the plane $\RR^2$. When taken modu\-lo~$\gL_\V$ the $s(\widetilde{e})t(\widetilde{e})$-diagonals and the $w(\widetilde{e})b(\widetilde{e})$-diagonals show the graphs $\gGL$ and $\gGLD$ embedded in the torus $\TT_\V$ and the duality between them.

\begin{Example}\label{exa:dual pictures}The two ways of putting diagonals in the left-hand picture in Fig.~\ref{fig:see quadrangles} yield the two pictures in Fig.~\ref{fig:classic1}. Both ways of putting diagonals in the right-hand picture in Fig.~\ref{fig:see quadrangles} lead to triangulations equivalent with the left-hand picture in Fig.~\ref{fig:classic2}.
\end{Example}

\begin{Example}\label{eq:exa:kagome} The Zhegalkin zebra function shown on the left in Fig.~\ref{fig:kagome} has no fractional matchings because $|\ospw|\neq |\ospb|$. Nonetheless if one takes the barycentres of the polygons, one finds the tiling by quadrangles as shown on the right in Fig.~\ref{fig:kagome}. The two ways of putting diagonals lead to respectively the left-hand picture in Fig.~\ref{fig:kagome} and the right-hand picture in Fig.~\ref{fig:classic2}.
\end{Example}

\begin{figure}[h!]\centering
\includegraphics{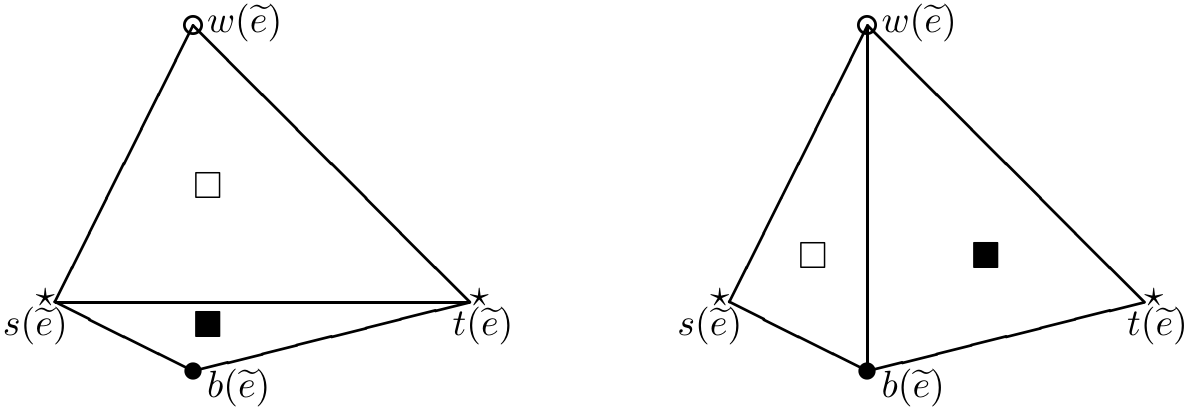}
\caption{Quadrangle for edge $\widetilde{e}$.}\label{fig:quadrangle}
\end{figure}

\begin{figure}[t]\centering
\includegraphics[width=5.5cm]{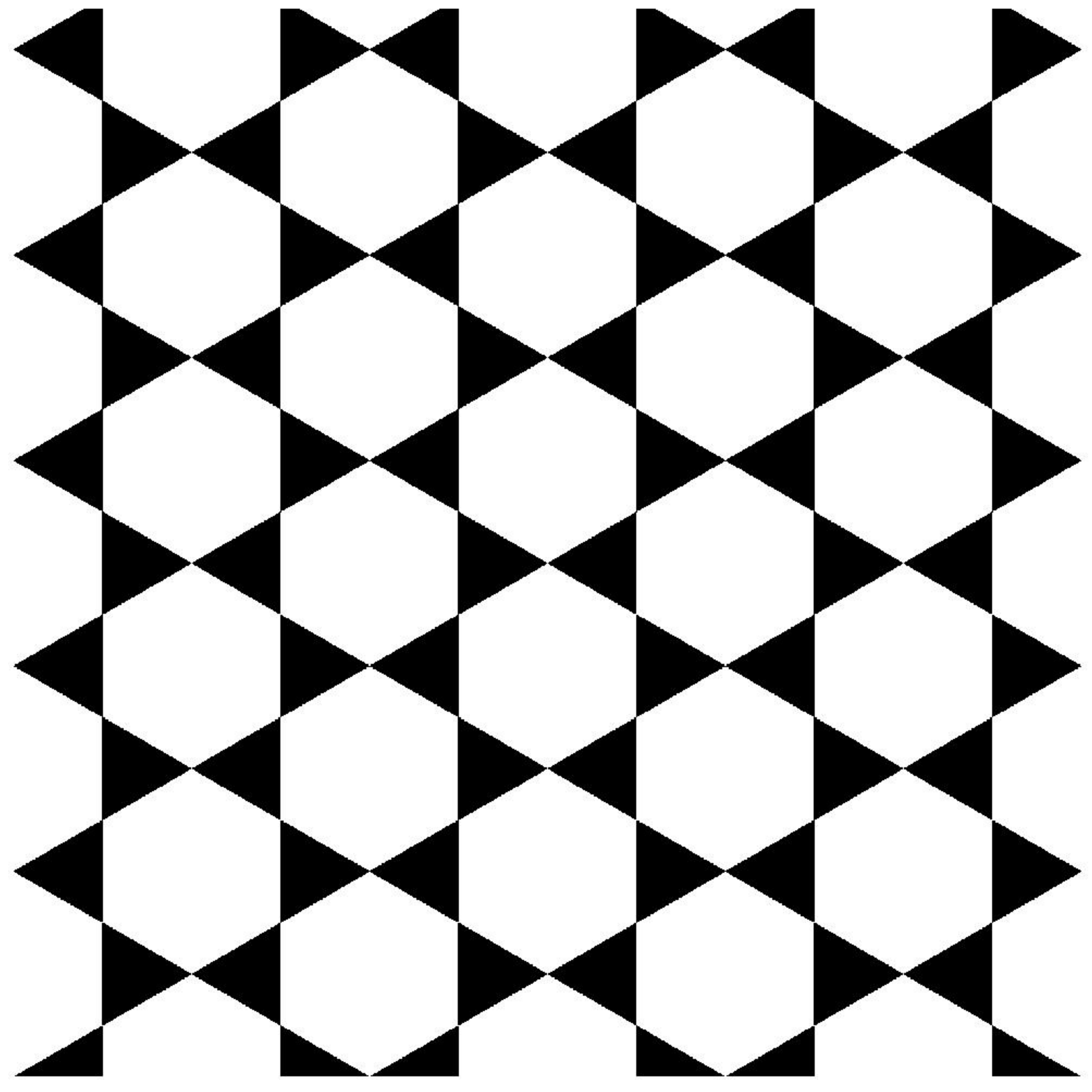}\qquad \includegraphics[width=6cm]{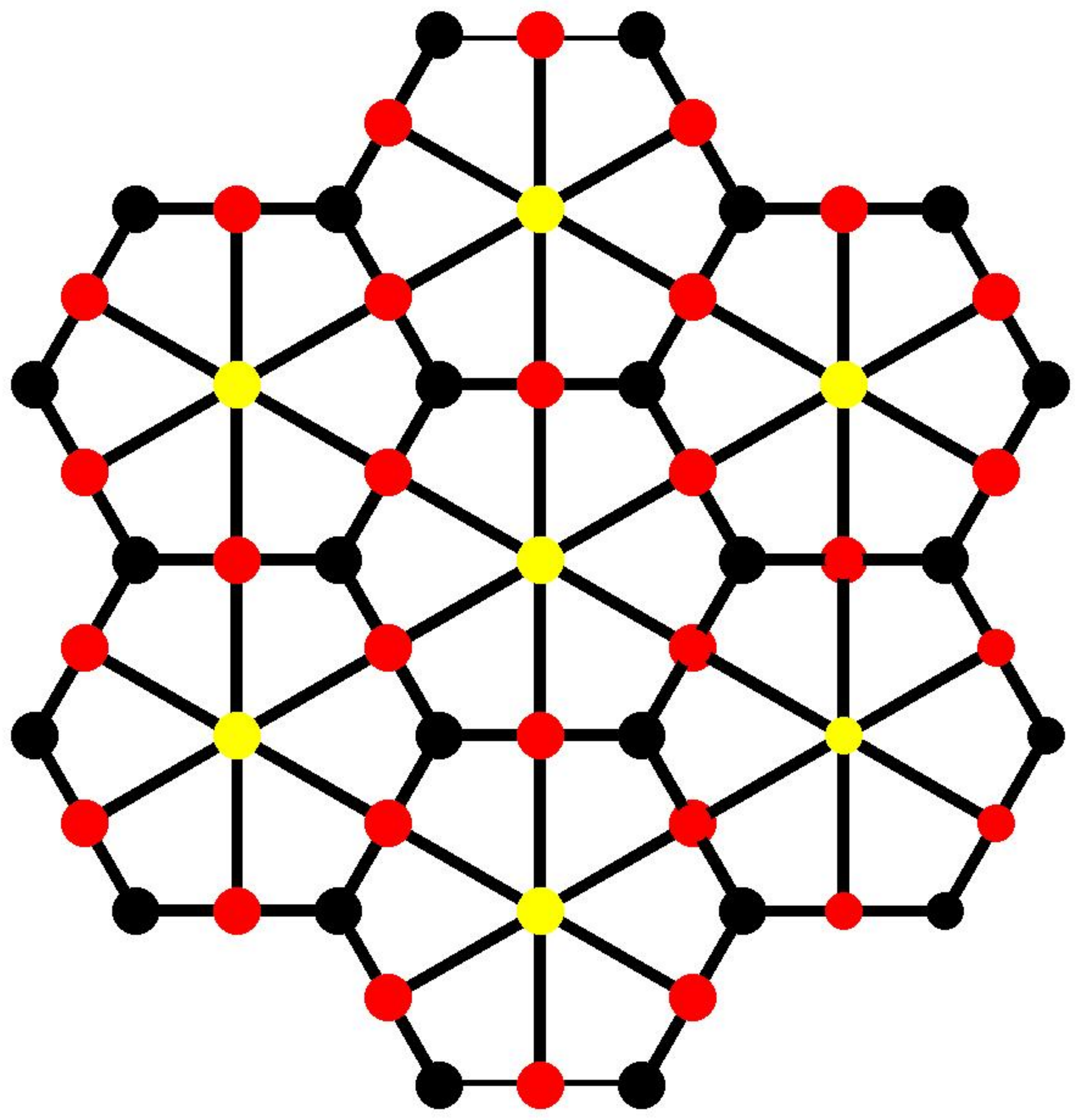}

$\cF=\ze{21}+\ze{41}+\ze{61}+\ze{62}$, $\gs_0=(1,4,5,3,2,6)$, $\gs_1=(1,2,5)(3,4,6)$.

\caption{A Zhegalkin zebra function without fractional matchings.}\label{fig:kagome}
\end{figure}

Each of its two diagonals divides a quadrangle into two triangles which we color black/white as indicated in Fig.~\ref{fig:quadrangle}. When the quadrangles are put together to make a tiling of the plane the colored triangles for the $s(\widetilde{e})t(\widetilde{e})$-diagonals fuse so as to form the black and white polygons in a tiling which we want to think of as \textit{the deformation, determined by $(\nu_1,\nu_2,\nu_3)$, of the tiling given by the Zhegalkin zebra function~$\cF$}.

The weight realization $(\nu_1,\nu_2,\nu_3)$ itself can be deformed by
\begin{gather}\label{eq:deform weights}
(\nu_1,\nu_2,\nu_3) \rightsquigarrow \big(N\nu_1+\nu'_1-\nu''_1, N\nu_2+\nu'_2-\nu''_2, N\nu_3+\nu'_3-\nu''_3\big),
\end{gather}
where $\nu'_1,\nu'_2,\nu'_3,\nu''_1,\nu''_2,\nu''_3\in\ocW$ are such that $\deg\nu'_j=\deg\nu''_j$ for $j=1,2,3$ and $N\in\ZZ_{\geq0}$ is so large that the positivity and strict convexity conditions are satisfied for the deformed triple.

The colored triangles for the $w(\widetilde{e})b(\widetilde{e})$-diagonals, on the other hand, make up a tiling of the plane $\RR^2$ by black and white triangles such that each triangle has one $\star$-vertex, one $\bullet$-vertex and one $\circ$-vertex. It is a well-known~\cite{LZ} that from such a triangulation one can construct a~branched covering $\cB\colon \TT_\V \rightarrow \CC\PP^1$ with precisely three branch points $0$, $1$, $\infty$:
\begin{gather*}
\cB\colon \ \TT_\V \longrightarrow \CC\PP^1 ,\qquad \cB(\ospw)=0,\qquad \cB(\ospb)=1,\qquad \cB(\ospv)=\infty .
\end{gather*}
This is where \textit{Zhegalkin zebra motives} meet \textit{dessins d'enfants}. In the works on dessins d'enfants on Riemann surfaces of genus~$1$ one wants to find on the torus a structure of elliptic curve over a number field such that the branched covering map is a morphism of varieties, called a~\textit{Belyi map}. We will not elaborate on dessins d'enfants, but refer instead to~\cite{LZ,Sc,VH}.

The map $\cB$ induces unramified coverings of $\CC\setminus\{0,1\} = \CC\PP^1\setminus\{0,1,\infty\}$:
\begin{gather*}
\TT_\V\setminus\bigl(\ospw,\ospb,\ospv\bigr) \stackrel{\cB}{\longrightarrow} \CC\setminus\{0,1\} \stackrel{\widetilde{\cB}}{\longleftarrow} \RR^2\setminus\bigl(\spw,\spb,\spv\bigr).
\end{gather*}
One can normalize the formulas describing $\widetilde{\cB}$ such that the $s(\widetilde{e})t(\widetilde{e})$-diagonals of the quadrangles are mapped to the line $\Re z=\breuk{1}{2}$ in $\CC$ while the midpoints of these diagonals are mapped to the point~$\breuk{1}{2}$. Every path in $\CC\setminus\{0,1\}$ starting at the point $\breuk{1}{2}$ can be lifted uniquely to a collection of paths in $\RR^2\setminus\bigl(\spw,\spb,\spv\bigr)$ starting at the midpoints of the $s(\widetilde{e})t(\widetilde{e})$-diagonals.

The fiber $\cB^{-1}\big(\breuk{1}{2}\big)$ can be identified with the set $\ocE$. \textit{The monodromy action of the fundamental group $\pi_1\bigl(\CC\setminus\{0,1\}, \breuk{1}{2}\bigr)$ on $\ocE$ is then exactly the permutation action described by the superpotential $\SFL=(\ocE,\gs_0,\gs_1)$.} This is illustrated in Fig.~\ref{fig:C01}.

\looseness=1 Lifting the figure-$\infty$-loop which starts at $\breuk{1}{2}$ in direction NW yields a collection of paths known as \textit{zigzags}. It is evident from Fig.~\ref{fig:C01} that these correspond to the orbits of the permuta\-tion~$\gs_1\gs_0$.

\begin{figure}[h!]\centering
\includegraphics{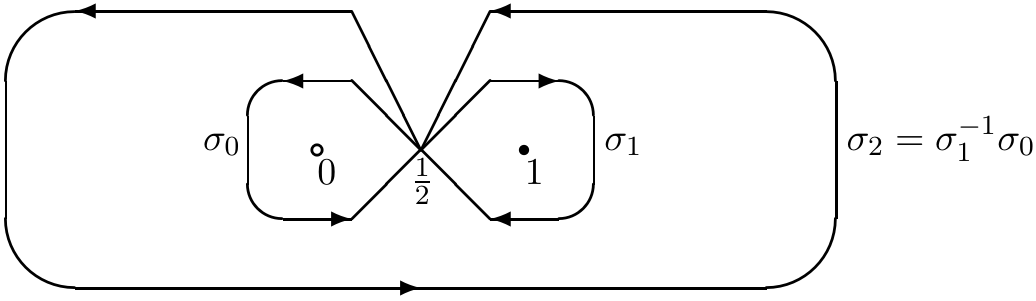}
\caption{$\PP^1\setminus\{0,1,\infty\} = \CC\setminus\{0,1\}$.}\label{fig:C01}
\end{figure}

Zigzag paths are visible in the quadrangle tiling as paths which enter and leave the quadrangles at opposite sides; see Figs.~\ref{fig:see quadrangles} and~\ref{fig:kagome}. Through every quadrangle pass two zigzag paths which one can let over/under cross as indicated in the left-hand picture in Fig.~\ref{fig:quadrangle3}. The zigzag paths form an \textit{alternating strand diagram} in the sense of~\cite[Definition~1.10]{Bo}.

Zigzag paths play an important role in the literature on dimer models and are used to formulate \textit{consistency conditions} \cite{Bo,Br,DHP,G,HS}. Because not all dimer models which come from Zhegalkin zebra functions do satisfy these consistency conditions we will not say more about zigzag paths.

Instead we focus on the lifts of the arrows $0\leftarrow\breuk{1}{2}$ and $\breuk{1}{2}\rightarrow 1$ and $\breuk{1}{2} \uparrow \infty$ shown in the right-hand picture in Fig.~\ref{fig:quadrangle3} as the vectors $\sfq_\bw(e)$, $\sfq_\bb(e)$ and $\sfq_\bt(e)$, respectively. Obviously, $\sfq_\bt(e)=\frac{1}{2}\V(e)$. Proposition \ref{prop:quadrangle} and formula~\eqref{eq:sfq} explicitly give the vectors $\sfq_\bw(e)$ and~$\sfq_\bb(e)$.

\begin{figure}[t]\centering
\includegraphics{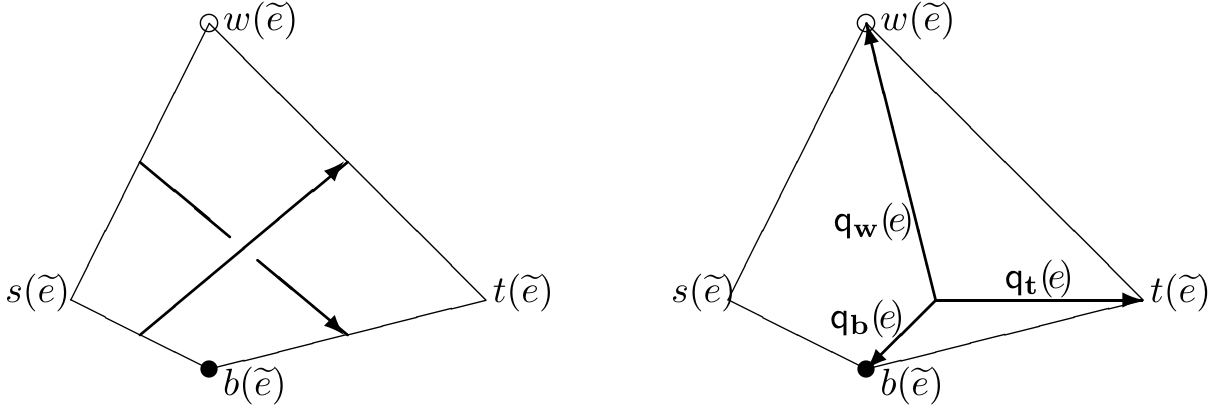}
\caption{Quadrangle for edge $e$.}\label{fig:quadrangle3}
\end{figure}

\textbf{Plan of the paper.} In Sections \ref{sec:geoweights} and \ref{sec:Jacobi} we give more details on, respectively, the \textit{toric geometry} and the \textit{non-commutative algebraic geometry} behind Section~\ref{subsec:intro superpotential}. In Section \ref{sec:mirror symmetry} we attempt to put the results of Sections~\ref{sec:geoweights} and~\ref{sec:Jacobi} into a \textit{mirror symmetry} perspective. Section~\ref{sec:practical} describes how one can solve some practical matters (by computer). Formula~\eqref{eq:evaluate} in Section~\ref{subsec:draw} is basically computer code for evaluating the Zhegalkin zebra function~$\cF$ and drawing the picture of the tiling. In Section~\ref{subsec:compot} I describe how one can compute the superpotential $\SFL$ and the realization~$\V_\cF$. From the superpotential one can easily determine all perfect matchings. It is described in the text between Definitions~\ref{def:complete} and~\ref{def:weight realization} how to obtain from this a weight realization, which subsequently can be deformed with formula~\eqref{eq:deform weights}. In Section~\ref{subsec:comlatquad} I describe how one can construct and draw the quadrangle tiling of $\RR^2$ for a given weight realization.

\section{The algebraic geometry of weights}\label{sec:geoweights}

In this section $\cF$ is a convex Zhegalkin zebra function and $\gL$ is a sublattice of $\auto(\cF)$, such that the superpotential $\SFL$ is dimer complete and such that a weight realization of the superpotential exists; see Definitions \ref{def:zzf}, \ref{def:superpotential}, \ref{def:complete} and~\ref{def:weight realization}. So there is a tiling of $\RR^2$ by strictly convex quadrangles and the diagonals provide embeddings of the graphs $\gGL$ and $\gGLD$ into the torus $\TT=\RR^2/\gL$.

\subsection[The (co)homology of $\gGL$ and $\gGLD$]{The (co)homology of $\boldsymbol{\gGL}$ and $\boldsymbol{\gGLD}$}\label{subsec:coho}

\pph{}\label{sec:superpotential3} Elements $\bv\in\ospv$, $\bb\in\ospb$, $\bw\in\ospw$ define maps $\ga_\bv, \gb_\bb, \gb_\bw\colon \ocE\longrightarrow\ZZ$ such that
\begin{gather}\label{eq:vertex cycle}
\ga_\bv(e) = 1 \quad \text{if} \quad t(e)=\bv ,\qquad \ga_\bv(e)=-1 \quad \text{if}\quad s(e)=\bv ,\qquad \ga_\bv(e)=0 \quad \text{else},\\
\label{eq:black cycle}
\gb_\bb(e) = 1\quad \text{if} \quad b(e)=\bb ,\qquad \gb_\bb(e)=0\quad \text{else},\\
\label{eq:white cycle}
\gb_\bw(e) = 1\quad \text{if} \quad w(e)=\bw ,\qquad \gb_\bw(e)=0\quad \text{else}.
\end{gather}
There is only one linear relation between the maps $\ga_\bv$ ($\bv\in\ospv$) and there is only one linear relation between the maps $\gb_\bb$, $\gb_\bw$ ($\bb\in\ospb$, $\bw\in\ospw$), namely
\begin{gather}\label{eq:ab relations}
\sum_{\bv\in\ospv}\ga_\bv = 0 ,\qquad\sum_{\bb\in\ospb}\gb_\bb = \sum_{\bw\in\ospw}\gb_\bw .
\end{gather}

\pph{} The first cohomology group $\ho^1(\gGL,\ZZ)$ of the graph $\gGL$ is the subgroup of the group $\ZZ^{\ocE}=\map(\ocE,\ZZ)$ consisting of the maps $\eta\colon \ocE\longrightarrow\ZZ$ which satisfy
\begin{gather}\label{eq:G homology cycle}
\sum_{e\in\ocE}\!\eta(e)\ga_\bv(e) = 0\qquad\textrm{for all}\quad \bv\in\ospv .
\end{gather}
The rank of this group is
\begin{gather*}
\rank \ho^1(\gGL, \ZZ) = |\ocE|-|\ospv|+1 = |\ospb|+|\ospw|+1 .
\end{gather*}
The maps $\gb_\bb$ and $\gb_\bw$ for $\bb\in\ospb$, $\bw\in\ospw$ are elements of $\ho^1(\gGL,\ZZ)$. They generate a subgroup of rank $|\ospb|+|\ospw|-1$; see \eqref{eq:black cycle}, \eqref{eq:white cycle} and \eqref{eq:ab relations}.

The embedding $\gGL\hookrightarrow\TT$ induces a homomorphism of homology groups
\begin{gather*}
\ho_1(\gGL, \ZZ)\longrightarrow\ho_1(\TT, \ZZ).
\end{gather*}
This homomorphism is surjective and its kernel is generated by the elements{\samepage
\begin{gather}\label{eq:face}
\cgb_\bb = \sum_{e\in\ocE}\gb_\bb(e) e\qquad\textrm{and}\qquad \cgb_\bw = \sum_{e\in\ocE}\gb_\bw(e) e
\end{gather}
for $\bb\in\ospb$ and $\bw\in\ospw$.}

\medskip

\pph{} A \textit{path of length $k$ on $\gGL$} is a sequence $\bp=(e_1,\ldots,e_k)$ in $\ocE$ such that $t(e_i)=s(e_{i+1})$ for $i=1,\ldots,k-1$. We define $s(\bp)=s(e_1)$, $t(\bp)=t(e_k)$. A~path $\bp$ on $\gGL$ is \textit{closed} if $s(\bp)=t(\bp)$. In case $k=0$ these are the constant paths supported on the vertices of $\gGL$.

The \textit{homology class} of a closed path $\bp=(e_1,\ldots,e_k)$ is
\begin{gather}\label{eq:loop class}
\ubp=e_1+\dots+e_k .
\end{gather}
$\ho_1(\gGL,\ZZ)$ is generated by the homology classes of closed paths on $\gGL$. Special closed paths on~$\gGL$ are given by the boundaries of the black and white polygons in the tiling. Their homology classes are $\cgb_\bb$ and $\cgb_\bw$ as in~\eqref{eq:face}. They generate a subgroup in $\ho_1(\gGL,\ZZ)$ of rank $|\ospb|+|\ospw|-1$.

\medskip

\pph{} The first cohomology group $\ho^1\big(\gGLD,\ZZ\big)$ of the graph $\gGLD$ is the subgroup of~$\ZZ^{\ocE}$ consisting of the maps $\theta\colon \ocE\longrightarrow\ZZ$ which satisfy
\begin{gather}\label{eq:homology cycle}
\forall\, \bw\in\ospw , \ \forall\, \bb\in\ospb \colon \quad \sum_{e\in\ocE,w(e)=\bw} \theta(e) = \sum_{e\in\ocE,b(e)=\bb} \theta(e) = 0 .
\end{gather}
Since there is only one linear relation between the equations in the system \eqref{eq:homology cycle} the rank of the cohomology group is
\begin{gather*}
\rank \ho^1\big(\gGLD, \ZZ\big) = |\ocE|-|\ospb|-|\ospw|+1 = |\ospv|+1 .
\end{gather*}
The maps $\ga_\bv$ for $\bv\in\ospv$ are elements of $\ho^1\big(\gGLD,\ZZ\big)$. They generate a subgroup of rank $|\ospv|-1$; see \eqref{eq:vertex cycle} and~\eqref{eq:ab relations}. The embedding $\gGLD\hookrightarrow\TT$ induces a homomorphism of homology groups
\begin{gather}\label{eq:D graph in torus}
\ho_1\big(\gGLD, \ZZ\big)\longrightarrow\ho_1(\TT, \ZZ).
\end{gather}
This homomorphism is surjective and its kernel is generated by the elements
\begin{gather}\label{eq:dual face}
\cga_\bv = \sum_{e\in\ocE}\ga_\bv(e)e\qquad\textrm{for}\quad\bv\in\ospv .
\end{gather}

\subsection[The geometry of $\protect{\Proj(\ZZ[\ocW])}$]{The geometry of $\boldsymbol{\Proj(\ZZ[\ocW])}$}\label{subsec:matching polytope}
In this section we investigate the geometry of the projective scheme $\Proj(\ZZ[\ocW])$, which by general constructions in algebraic geometry is associated with the graded semi-group $\ocW$ of integer weight functions for $\SFL$; see \cite[Chapter~II]{Ha}.

\medskip

\pph{}\label{pph:WGLD} By Definition \ref{def:weight} an integer weight function for the superpotential $\SFL$ is a map $\nu\colon \ocE\rightarrow\ZZ_{\geq0}$ which satisfies
\begin{gather}\label{eq:weight condition}
\forall\, \bw\in\ospw ,\ \forall\, \bb\in\ospb\colon \quad \sum_{w(e)=\bw} \nu(e) = \sum_{b(e)=\bb} \nu(e) = \deg \nu .
\end{gather}
From \eqref{eq:homology cycle} and \eqref{eq:weight condition} one sees that the difference $\nu-\nu'$ of two weight functions with the same degree is an element of~$\ho^1\big(\gGLD,\ZZ\big)$.

Conversely, if $\nu$ is a \textit{positive weight function} (i.e., $\nu(e)>0$ for all $e\in\ocE$) and $\theta$ is an element of~$\ho^1\big(\gGLD,\ZZ\big)$, then for all sufficiently large integers $N$ the function $N\nu-\theta$ is a positive weight function. Thus we find that
\begin{gather}\label{eq:rank W}
\rank \ocW = \rank \ho^1\big(\gGLD, \ZZ\big)+1 = |\ospv|+2 .
\end{gather}

\pph{}\label{pph:equivalence} On the semi-group $\ocW$ we define an equivalence relation $\sim$ by
\begin{gather}\label{eq:equivalence matchings}
\nu\sim\nu' \Leftrightarrow \exists\, r\colon \ \ospv\rightarrow\ZZ\quad\textrm{s.t.}\quad\nu-\nu'= \sum_{\bv\in\ospv} r_\bv\ga_\bv
\end{gather}
with $\ga_\bv$ as in \eqref{eq:vertex cycle}. We denote the set of equivalence classes by $\oocW$:
\begin{gather}\label{eq:Wsim}
\oocW = \ocW/{\sim}.
\end{gather}
This is a graded semi-group of rank $3$. The natural surjective homomorphism of semi-groups $\ocW\rightarrow\oocW$ is the analogue of the surjective homomorphism of groups $\ho_1\big(\gGLD, \ZZ\big)\rightarrow\ho_1(\TT, \ZZ)$ induced by the embedding $\gGLD\hookrightarrow\TT$; cf.~\eqref{eq:D graph in torus}.

\medskip

\pph{} Recall from Definition \ref{def:weight} that the integer weight functions of degree $1$ are called perfect matchings and that~$\ocM$ is the set of perfect matchings. We denote the set of equivalence classes for the relation~$\sim$ on $\ocM$ by~$\ocA$:
\begin{gather*}
\ocA = \ocM/{\sim}.
\end{gather*}

\begin{Definition}[{cf.~\cite[Sections~3.4 and~3.5]{GK}}]\label{def:match polytope} The convex hull $\conv(\ocM)$ of $\ocM$ in $\RR^{\ocE}$ is called the \emph{matching polytope of $\SFL$}. The elements of $\conv(\ocM)$ are called \emph{fractional matchings}. The convex hull $\conv(\ocA)$ of~$\ocA$ is called the \emph{Newton polygon of~$\SFL$}.
\end{Definition}

\begin{Proposition}[{cf.~\cite[Lemma 3.10]{GK}}]\label{prop:matching polytope}\quad
\begin{enumerate}\itemsep=0pt
\item[$(i)$] The matching polytope satisfies\begin{gather}\label{eq:matching polytope}
\conv(\ocM) = \left\{\begin{array}{@{}l@{}}
\text{maps $\theta\colon \ocE\longrightarrow\RR_{\geq0}$ s.t.\ for all $\bw\in\ospw$ and}\\
\textrm{all } \bb\in\ospb\colon \sum\limits_{e\in\bw} \theta(e) = \sum\limits_{e\in\bb} \theta(e) = 1
\end{array}\right\}.
\end{gather}
The set of its vertices is precisely the set of perfect matchings $\ocM$.
\item[$(ii)$]
The semi-group $\ocW$ is generated by the perfect matchings and the semi-group $\oocW$ is ge\-ne\-rated by the set $\ocA$:
\begin{gather}\label{eq:weight span}
\ocW = \ZZ_{\geq0}\ocM ,\qquad \oocW = \ZZ_{\geq0}\ocA .
\end{gather}
\item[$(iii)$]
The matching polytope $\conv(\ocM)$ has dimension $|\ospv|+1$ and the Newton polygon $\conv(\ocA)$ has dimension $2$.
\end{enumerate}
\end{Proposition}
\begin{proof} $(i)$ This is Lemma 3.10 in \cite{GK}.

$(ii)$ Let $\nu\in\ocW$, $\nu\neq\nul$, be given. By $(i)$ there are non-negative real numbers $r_{\nu,\sm}$, $\sm\in\ocM$, such that
\begin{gather*}
\nu =\sum_{\sm\in\ocM} r_{\nu,\sm} \sm ,\qquad \sum_{\sm} r_{\nu,\sm} = \deg \nu.
\end{gather*}
Then $\nu(e)\geq r_{\nu,\sm} \sm(e)$ for all $\sm$ and $e$. Now take $\sm$ such that $r_{\nu,\sm}>0$. Then $\nu(e)-\sm(e)\geq 0$ for all $e\in\ocE$. This means that $\nu-\sm\in\ocW$. Note that $\deg(\nu-\sm)=\deg\nu-1$. If $\nu-\sm\neq\nul$ we repeat the preceding step with $\nu-\sm$ instead of $\nu$. After finitely many steps we arrive at the situation that $\nu$ minus some linear combination of perfect matchings with positive integer coefficients is~$0$. This result passes well to $\sim$-equivalence classes.

$(iii)$ This follows from \eqref{eq:rank W}, \eqref{eq:matching polytope} and \eqref{eq:weight span}.
\end{proof}

\pph{} General constructions in algebraic geometry associate with the graded semi-groups~$\ocW$ and~$\oocW$ the projective schemes $\Proj(\ZZ[\ocW])$ and $\Proj(\ZZ[\oocW])$; see \cite[Chapter~II]{Ha}. The natural surjective homomorphism of semi-groups $\ocW\rightarrow\oocW$ becomes an inclusion as closed subscheme
\begin{gather*}
\Proj\big(\ZZ[\oocW]\big) \subset \Proj(\ZZ[\ocW]) .
\end{gather*}

\pph{} As a consequence of \eqref{eq:weight span} we have a surjective homomorphism of rings
\begin{gather}\label{eq:weight ring}
\ZZ[X_\sm\,|\,\sm\in\ocM]\longrightarrow\ZZ[\ocW] ,\qquad X_\sm\mapsto X^\sm
\end{gather}
from the polynomial ring in the variables $X_\sm$, $\sm\in\ocM$, to the semi-group ring of $\ocW$; here~$X^\sm$ denotes the element of $\ZZ[\ocW]$ which corresponds to $\sm\in\ocW$. The kernel of the homomorphism~\eqref{eq:weight ring} is the ideal generated by the polynomials
\begin{gather}\label{eq:binomials}
\prod_{\sm\colon \mu(\sm)>0} X_{\sm}^{\mu(\sm)} - \prod_{\sm\colon \mu(\sm)<0} X_{\sm}^{-\mu(\sm)}\qquad\textrm{for}\quad \mu\in\MM,
\end{gather}
where $\MM$ denotes the lattice of $\ZZ$-linear relations between the perfect matchings:
\begin{gather}\label{eq:matchrel}
\MM=\bigg\{\mu\in\ZZ^{\ocM}\,|\, \sum_{\sm\in\ocM}\mu(\sm)\sm=0\bigg\}.
\end{gather}
It follows that $\Proj(\ZZ[\ocW])$ can be identified with the closed subscheme of the projective space
$\PP^{|\ocM|-1}=\Proj(\ZZ[X_\sm\,|\,\sm\in\ocM])$ given by the homogeneous equations
\begin{gather}\label{eq:binomial eqs}
\prod_{\sm\colon \mu(\sm)>0} X_{\sm}^{\mu(\sm)} = \prod_{\sm\colon \mu(\sm)<0} X_{\sm}^{-\mu(\sm)}\qquad\textrm{for}\quad \mu\in\MM.
\end{gather}

\pph{}\label{pph:points} A perfect matching $\sm_0$ defines an open subscheme of $\Proj(\ZZ[\ocW])$, namely the affine scheme
$\Spec\big(\ZZ\big[\ocW^0[-\sm_0]\big]\big)$ given by the semi-group
\begin{gather*}
\ocW^0[-\sm_0] = \big\{\nu-(\deg\nu)\sm_0\in\ZZ^{\ocE}\,|\,\nu\in\ocW\big\}
\end{gather*}
(cf.~\cite[Chapter~II, Proposition~2.5]{Ha}). This is a sub-semi-group of $\ho_1\big(\gGLD, \ZZ\big)$. The schemes $\Spec\big(\ZZ\big[\ocW^0[-\sm_0]\big]\big)$ for $\sm_0\in\ocM$ form a covering of $\Proj(\ZZ[\ocW])$ by affine open subschemes. Their intersection is $\Spec\big(\ZZ\big[\ho_1\big(\gGLD, \ZZ\big)\big]\big)$.

\begin{Proposition}\label{prop:subtori} The following diagram is commutative
\begin{gather*}
\begin{array}{@{}c@{\,\,}c@{\,\,}c@{}}
\ho^1(\TT, \CC^*)&\hookrightarrow&\Proj\big(\ZZ[\oocW]\big)
\\
\downarrow&&\downarrow
\\
\ho^1\big(\gGLD, \CC^*\big)&\hookrightarrow&\Proj(\ZZ[\ocW]).
\end{array}
\end{gather*}
The $(|\ospv|+1)$-dimensional complex torus $\ho^1\big(\gGLD, \CC^*\big)$ is an open subset in the set of complex points of the $(|\ospv|+1)$-dimensional scheme $\Proj(\ZZ[\ocW])$.

The $2$-dimensional complex torus $\ho^1(\TT, \CC^*)$ is an open subset in the set of complex points of $2$-dimensional scheme $\Proj(\ZZ[\oocW])$.
\end{Proposition}

\begin{Remark} The semi-group $\ocW^0[-\sm_0]$ is generated by the elements $\sm_1-\sm_0$ with $\sm_1$ a~perfect matching $\neq\sm_0$. By Proposition~\ref{prop:matching polytope} these are precisely the vectors along the edges of the matching polytope $\conv(\ocM)$ incident to the vertex~$\sm_0$. Therefore the duals of the semi-groups $\ocW^0[-\sm_0]$ for $\sm_0\in\ocM$ give precisely the maximal cones in the fan associated with the matching polytope $\conv(\ocM)$ by the construction in \cite[p.~26]{Fu}. Thus $\Proj(\ZZ[\ocW])$ can also be obtained with standard toric geometry constructions from the fan of outward pointing vectors to the matching polytope $\conv(\ocM)$.
\end{Remark}

\pph{}\label{pph:GK} We recall from \cite{GK} the construction of the skew symmetric bilinear form $\vge$ on $\ho_1\big(\gGLD, \ZZ\big)$ which gives the Poisson structure on the group ring $\ZZ\big[\ho_1\big(\gGLD, \ZZ\big)\big]$. In order to facilitate the exposition we reproduce Fig.~38 and formula~(65) of~\cite{GK} in our Fig.~\ref{fig:GK figure}. The definition, in \cite[Lemma~8.1 and Fig.~38]{GK}, of the local pairing at a white node $\bw$ of the graph $\gGLD$ can be phrased as follows. Let $e$, $e'$ and $e''$ in $\ocE$ be such that $w(e)=w(e')=w(e'')=\bw$. Let $\vec{E}$, $\vec{E}'$, $\vec{E}''$ be the edges of the graph $\gGLD$ dual to $e$, $e'$, $e''$, respectively, and pointing away from the vertex $\bw$. Write the cycle of $\gs_0$ which corresponds to $\bw$ as $(e_1,\ldots,e_q)$ with $e_1=e$ and let $e'=e_j$ and $e''=e_h$. Then formula~(65) in~\cite{GK} can be stated as
\begin{gather}\label{eq:GK65w}
\delta_\bw\bigl(\big(\vec{E}'-\vec{E}\big)\wedge\big(\vec{E}''-\vec{E}\big)\bigr) = \breuk{1}{2}\sign(h-j) .
\end{gather}
A similar formula holds for the local pairing at a black node $\bb$ of $\gGLD$, but since in our convention the boundaries of the black polygons are oriented clockwise, there is an extra $-$-sign:
\begin{gather}\label{eq:GK65b}
\delta_\bb\bigl(\big(\vec{E}'-\vec{E}\big)\wedge\big(\vec{E}''-\vec{E}\big)\bigr) = -\breuk{1}{2}\sign(h-j) .
\end{gather}
Definition~8.2 in~\cite{GK} builds the skew symmetric bilinear form $\vge$ on $\ho_1\big(\gGLD, \ZZ\big)$ from these local pairings. For reasons that will become clear in \eqref{eq:form m 5} we denote this form as $\gep$. The defining formula in \cite{GK} can then be stated as
\begin{gather}\label{eq:GK8.2}
\gep = \sum_{\bw\in\ospw}\delta_\bw - \sum_{\bb\in\ospb}\delta_\bb .
\end{gather}

\begin{figure}[t]\centering
\includegraphics{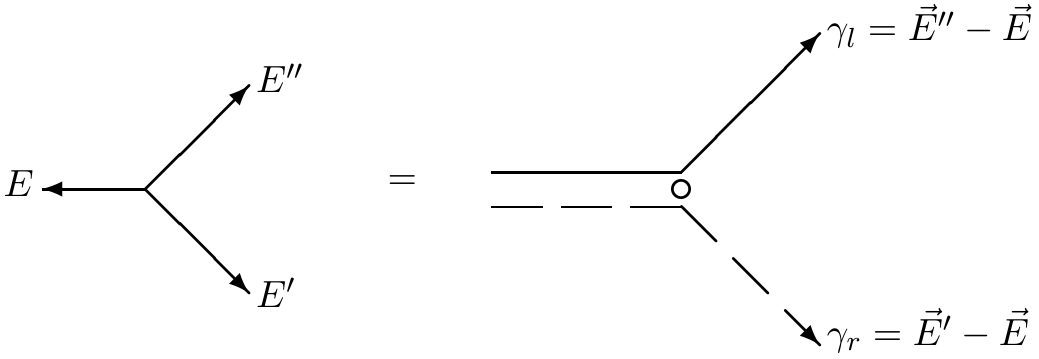}
\caption{Fig.~38 + formula (65) in \cite{GK}: One has $\delta_v(\gamma_r,\gamma_l)=\frac{1}{2}$.}\label{fig:GK figure}
\end{figure}

By \cite[Section~1.1]{GK} the \textit{Poisson bracket} on $\ZZ\big[\ho_1\big(\gGLD, \ZZ\big)\big]$ is then given by
\begin{gather}\label{eq:poisson bracket1}
\bigl\{X^{L_1}, X^{L_2}\bigr\}=\gep(L_1,L_2) X^{L_1+L_2} ,
\end{gather}
where $X^{L}$ denotes the element of $\ZZ\big[\ho_1\big(\gGLD, \ZZ\big)\big]$ which corresponds to the (homology class of) the loop $L$ on $\gGLD$.

\medskip

\pph{}\label{pph:Poisson3} We are now going to give a simple description of the form $\gep$ \eqref{eq:GK8.2} in terms of the permutations $\gs_0$ and $\gs_1$. Fix a perfect matching $\sm$. Write $\gs_0$ and $\gs_1$ as permutation matrices; i.e., matrices with rows and columns indexed by the elements of $\ocE$ and in column $e$ only one non-zero entry, namely $1$ in row $\gs_0(e)$ (resp.~$\gs_1(e)$). By multiplying for each
$e\in\ocE$ the corresponding column by~$1-\sm(e)$ we obtain two new matrices $\vgs_{\sm,0}$ and $\vgs_{\sm,1}$, respectively. These are nilpotent matrices. We set
\begin{gather}\label{eq:broken}
\rho_{\sm,0}=(\II-\vgs_{\sm,0})^{-1} ,\qquad \rho_{\sm,1}=(\II-\vgs_{\sm,1})^{-1} .
\end{gather}
The meaning of these matrices is as follows. Write $\gs_0$ as a product of disjoint cycles such that the elements of $\sm$ are in the first position in their cycle. Then for $e\neq e'\in\ocE$ the entry in column $e$ and row $e'$ in the matrix $\rho_{\sm,0}$ is $1$ if and only if $e$ and $e'$ are in the same cycle of $\gs_0$ with $e$ to the left of $e'$. And similarly for $\gs_1$ and $\rho_{\sm,1}$.

Let $\sm'$, $\sm''$ be two perfect matchings. Then $\sum\limits_{e\in\ocE}(\sm'(e)-\sm(e)) e$ and $\sum\limits_{e\in\ocE}(\sm''(e)-\sm(e)) e$ are elements of $\ho_1\big(\gGLD, \ZZ\big)$. It follows from \eqref{eq:GK65w}--\eqref{eq:GK8.2} that
\begin{gather*}
\gep(\sm'-\sm,\sm''-\sm) = \breuk{1}{2}{\sm'}^t\big(\rho_{\sm,0}-\rho_{\sm,0}^t+ \rho_{\sm,1}-\rho_{\sm,1}^t\big)\sm'',
\end{gather*}
where, for simplicity of notation, we have written on the left-hand side $\sm'-\sm$ and $\sm''-\sm$ for $\sum\limits_{e\in\ocE}(\sm'(e)-\sm(e)) e$ and $\sum\limits_{e\in\ocE}(\sm''(e)-\sm(e)) e$, respectively, while on the right-hand side we view~$\sm'$ and~$\sm''$ as column vectors.

By bilinearity this extends to all perfect matchings $\sm_1$, $\sm_2$, $\sm_3$, $\sm_4$:
\begin{gather}\label{eq:form m 2}
\gep(\sm_1-\sm_2,\sm_3-\sm_4) = \breuk{1}{2} (\sm_1-\sm_2)^t\big(\rho_{\sm,0}-\rho_{\sm,0}^t+\rho_{\sm,1}-\rho_{\sm,1}^t\big)(\sm_3-\sm_4).
\end{gather}
In the matrix $\rho_{\sm,0}+\rho_{\sm,0}^t-\II$ the $(e,e')$-entry is $1$ if $e$ and $e'$ sit in the same cycle of~$\gs_0$ and is $0$ otherwise. Consequently, $\bigl(\rho_{\sm,0}+\rho_{\rho,0}^t-\II\bigr)(\sm_3-\sm_4)=0$ and
\begin{gather}\label{eq:form m 3}
\breuk{1}{2}\bigl(\rho_{\sm,0}-\rho_{\sm,0}^t\bigr)(\sm_3-\sm_4)=\bigl(-\breuk{1}{2}\II+\rho_{\sm,0}\bigr)(\sm_3-\sm_4).
\end{gather}
Similarly
\begin{gather}\label{eq:form m 4}
\breuk{1}{2}\bigl(\rho_{\sm,1}-\rho_{\sm,1}^t\bigr)(\sm_3-\sm_4)=\bigl({-}\breuk{1}{2}\II+\rho_{\sm,1}\bigr)(\sm_3-\sm_4).
\end{gather}
Since $\ho_1\big(\gGLD, \ZZ\big)$ is the subgroup of $\ZZ^{\ocE}$ which is generated by the differences of pairs of perfect matchings we conclude from \eqref{eq:form m 2}--\eqref{eq:form m 4}:

\begin{Proposition}\label{prop:poisson m} The bilinear form $\gep$ on $\ho_1\big(\gGLD, \ZZ\big)$ is the restriction of the bilinear form on~$\ZZ^{\ocE}$ associated with the matrix $\rho_{\sm,0}+\rho_{\sm,1}-\II $:
\begin{gather}\label{eq:form m 5}
\forall\, \bh_1, \bh_2\in\ho_1\big(\gGLD, \ZZ\big)\colon \quad \gep(\bh_1,\bh_2) = \bh_1^t \bigl(\rho_{\sm,0}+\rho_{\sm,1}-\II\bigr) \bh_2.
\end{gather}
This holds for every perfect matching $\sm$.
\end{Proposition}

\pph{} The difference of the matrices $-\breuk{1}{2}\II+\rho_{\sm,0}$ and $-\breuk{1}{2}\II+\rho_{\sm,1}$ induces another anti-symmetric bilinear form $\gem$ on $\ho_1\big(\gGLD, \ZZ\big)$:
\begin{gather}\label{eq:form m 6}
\forall\, \bh_1, \bh_2\in\ho_1\big(\gGLD, \ZZ\big)\colon\quad \gem(\bh_1,\bh_2) = \bh_1^t \bigl(\rho_{\sm,0}-\rho_{\sm,1}\bigr) \bh_2.
\end{gather}
The form $\gem$ can also be defined with the method of Section~\ref{pph:GK}, i.e.~(cf.~\eqref{eq:GK8.2})
\begin{gather*}
\gem = \sum_{\bw\in\ospw}\delta_\bw + \sum_{\bb\in\ospb}\delta_\bb .
\end{gather*}
As a consequence the right-hand side of~\eqref{eq:form m 6} is independent of the choice of the perfect matching $\sm$.

\medskip

\pph{} The bilinear forms $\gep$ and $\gem$ in \eqref{eq:form m 5}--\eqref{eq:form m 6} define two Poisson structures on $\ZZ\big[\ho_1\big(\gGLD, \ZZ\big)\big]$; see~\eqref{eq:poisson bracket1}. So, they also define Poisson structures on $\ZZ[\ocW^0[-\sm_0]]$ for every $\sm_0\in\ocM$. Altogether this gives:

\begin{Theorem}\label{thm:poisson} The bilinear forms $\gep$, $\gem$ in \eqref{eq:form m 5}--\eqref{eq:form m 6}, i.e., for $\bh_1, \bh_2\in\ho_1\big(\gGLD, \ZZ\big)$
\begin{gather*}
\gep(\bh_1,\bh_2)=\bh_1^t \bigl(\rho_{\sm,0}+\rho_{\sm,1}-\II\bigr) \bh_2 ,\qquad
\gem(\bh_1,\bh_2)=\bh_1^t \bigl(\rho_{\sm,0}-\rho_{\sm,1}\bigr) \bh_2 ,
\end{gather*}
define two Poisson structures on $\Proj(\ZZ[\ocW])$ which extend the Poisson structures on the torus $\ho^1\big(\gGLD,\CC^*\big)$ described in {\rm \cite[Section~1.1]{GK}}.
\end{Theorem}

\pph{}\label{pph:symplectic structures} Note the equivalences of data
\begin{itemize}\itemsep=0pt
\item the pair of matrices $\rho_{\sm,0}+\rho_{\sm,1}-\II$ and $\rho_{\sm,0}-\rho_{\sm,1}$,
\item the pair of matrices $-\II+2\rho_{\sm,0}$ and $-\II+2\rho_{\sm,1}$,
\item the pair of permutations $\gs_0$ and $\gs_1$ plus the perfect matching $\sm$.
\end{itemize}
The four matrices define anti-symmetric bilinear forms on $\ho_1\big(\gGLD, \ZZ\big)$:
\begin{gather}\label{eq:poisp} \gep(\bh_1,\bh_2)=\bh_1^t\cdot(\rho_{\sm,0}+\rho_{\sm,1}-\II)\cdot\bh_2 ,\\
\label{eq:poism} \gem(\bh_1,\bh_2)=\bh_1^t\cdot(\rho_{\sm,0}-\rho_{\sm,1})\cdot\bh_2 ,\\
\label{eq:poisb} \geb(\bh_1,\bh_2)=\bh_1^t\cdot(-\II+2\rho_{\sm,1})\cdot\bh_2 ,\\
\label{eq:poisw} \gew(\bh_1,\bh_2)=\bh_1^t\cdot(-\II+2\rho_{\sm,0})\cdot\bh_2 ,
\end{gather}
for $\bh_1, \bh_2\in\ho_1\big(\gGLD, \ZZ\big)$; cf.~\eqref{eq:form m 5}--\eqref{eq:form m 6}. One has the obvious relations
\begin{gather*}
\gep=\breuk{1}{2}(\geb+\gew) ,\qquad \gem=\breuk{1}{2}(\gew-\geb) .
\end{gather*}
Each of the forms $\gep$, $\gem$, $\geb$, $\gew$ defines a Poisson structure on $\Proj(\ZZ[\ocW])$, independent of the choice of the perfect matching $\sm$.

The matrices $-\II+2\rho_{\sm,0}$ and $-\II+2\rho_{\sm,1}$ are of the form $\II+\mathrm{nilpotent}$ and are therefore invertible. Note that
\begin{gather}\label{eq:rho01}
-\II+2\rho_{\sm,0}=(\II+\vgs_{\sm,0})(\II-\vgs_{\sm,0})^{-1} ,\qquad -\II+2\rho_{\sm,1}=(\II+\vgs_{\sm,1})(\II-\vgs_{\sm,1})^{-1} .
\end{gather}
The matrix $\rho_{\sm,0}-\rho_{\sm,1}$ is not invertible, because for every $\bv\in\ospv$:
\begin{gather}\label{eq:form m 7}
\bigl(\rho_{\sm,0}-\rho_{\sm,1}\bigr) \ga_\bv = 0
\end{gather}
with $\ga_\bv$ as in \eqref{eq:vertex cycle}, here viewed as a column vector. The kernel of the homomorphism $\ho_1\big(\gGLD, \ZZ\big)$ $\rightarrow\ho_1(\TT,\ZZ)=\gL$ induced by the embedding $\gGLD\hookrightarrow\TT$, is generated by the elements $\cga_\bv$, $\bv\in\ospv$. So~\eqref{eq:form m 7} means that the bilinear form $\gem$ in~\eqref{eq:form m 6} is the pull-back of the intersection form on $\ho_1(\TT,\ZZ)$; see also~\cite[Section~1.1]{GK}.

\begin{Definition}\label{def:zigzag} \textup{A \emph{zigzag} for the superpotential $\SFL=(\ocE,\gs_0,\gs_1)$ is a cycle of the permuta\-tion~$\gs_1\gs_0$. The set of zigzags is denoted by $\ospz$. Every zigzag $\bz$ defines a map} $\ga_\bz\colon \ocE\longrightarrow\ZZ$,
\begin{gather*}
\ga_\bz(e)=1\quad\text{if}\quad e\in\bz,\qquad \ga_\bz(e)=-1\quad \textrm{if}\quad \gs_0(e)\in\bz,\qquad \ga_\bz(e)=0\quad\textrm{else}.
\end{gather*}
\end{Definition}

The matrix $\rho_{\sm,0}+\rho_{\sm,1}-\II$ is not invertible either, because, as one easily checks,
\begin{gather}\label{eq:form m 8}
\bigl(\rho_{\sm,0}+\rho_{\sm,1}-\II\bigr) \ga_\bz = 0
\end{gather}
for every zigzag $\bz$. Compare formula~\eqref{eq:form m 8} with~\cite[Lemma~1.1]{GK}.

\medskip

\pph{} The matrices $-\II+2\rho_{\sm,0}$ and $-\II+2\rho_{\sm,1}$ have entries in $\ZZ_{\geq0}$ and are of the form $\II+\mathrm{nilpotent}$. So, they give injective (but not surjective) homomorphisms of semi-groups
\begin{gather}\label{eq:Wduality maps}
\sT_{\sm,0}, \sT_{\sm,1}\colon \ \ocW\longrightarrow\WLD ,
\end{gather}
which depend on the choice of the perfect matching $\sm$. Here $\WLD$ is the semi-group dual to $\ocW$.

\subsection[The geometry of $\protect{\Spec(\ZZ[\ocW])}$]{The geometry of $\boldsymbol{\Spec(\ZZ[\ocW])}$}\label{subsec:specW}

\pph{}\label{pph:points 2} By definition complex points of the scheme $\Spec(\ZZ[\ocW])$ are ring homomorphisms $\ZZ[\ocW]\rightarrow\CC$, or equivalently, homomorphism of semi-groups $\xi\colon \ocW\rightarrow\CC^\times$, where $\CC^\times$ denotes the set $\CC$ with multiplication as binary operation. Such a homomorphism $\xi$ is completely determined by the complex numbers $\xi_\sm=\xi(\sm)$, $\sm\in\ocM$, which must satisfy the equations~\eqref{eq:binomial eqs}. So
\begin{gather*}
\CC\Spec(\ZZ[\ocW])=\bigg\{(\xi_\sm)\in\CC^{\ocM}\,\big|\,\forall\,\mu\in\MM\colon
\prod_{\mu(\sm)>0} \xi_{\sm}^{\mu(\sm)} = \prod_{\mu(\sm)<0} \xi_{\sm}^{-\mu(\sm)}\bigg\},
\end{gather*}
where $\MM$ denotes the lattice of $\ZZ$-linear relations between the perfect matchings:
\begin{gather*}
\MM=\bigg\{\mu\in\ZZ^{\ocM}\,\big|\, \sum_{\sm\in\ocM}\mu(\sm)\sm=0\bigg\}.
\end{gather*}
A map $\psi\colon \ocE\rightarrow\CC$ gives a point in $\CC\Spec(\ZZ[\ocW])$ through the homomorphism
\begin{gather}\label{eq:psi point}
E_\psi\colon \ \ocW\longrightarrow\CC^*\subset\CC^\times ,\qquad E_\psi(\nu) = \exp\bigg(\sum_{e\in\ocE}\nu(e)\psi(e)\bigg) .
\end{gather}
The following commutative diagram helps to locate these points
\begin{gather}\label{eq:specproj}
\begin{array}{@{}ccccccc}
\Spec(\ZZ[\oocW])&\rightarrow&\Spec(\ZZ[\ocW])&\rightarrow&\aA^{|\ocM|}&\supset&\CC^{|\ocM|}
\\[2ex]
\uparrow&&\uparrow&&\uparrow&&\uparrow
\\[2ex]
\Spec(\ZZ[\oocW])\setminus\nul&\rightarrow&
\Spec(\ZZ[\ocW])\setminus\nul&\rightarrow&\aA^{|\ocM|}\setminus\nul&\supset&(\CC^*)^{|\ocM|}
\\[2ex]
\downarrow&&\downarrow&&\downarrow&&\downarrow
\\[2ex]
\Proj(\ZZ[\oocW])&\rightarrow&\Proj(\ZZ[\ocW])&\rightarrow&\PP^{|\ocM|-1}&\supset&(\CC^*)^{|\ocM|}/\CC^*
\\[2ex]
\uparrow&&\uparrow&&\uparrow&\nearrow_\simeq&
\\[2ex]
\ho^1(\TT, \CC^*)&\rightarrow&\ho^1(\gGLD, \CC^*)&\rightarrow&(\CC^*)^{|\ocM|-1}&&
\end{array}
\end{gather}
where on the second line $\setminus\nul$ is shorthand for $\cap\big(\aA^{|\ocM|}\setminus\nul\big)$.

The point given by \eqref{eq:psi point} appears, for instance, in $\CC^{|\ocM|}$ as $\bigl(\exp(E_\psi(\sm))\bigr)_{\sm\in\ocM}$. Since its coordinates are $\neq0$ it projects down into $\ho^1\big(\gGLD, \CC^*\big) \subset (\CC^*)^{|\ocM|-1}$.

Replacing $\psi$ by $\psi+\sum\limits_{\bb\in\ospb} c_\bb\gb_\bb +\sum\limits_{\bw\in\ospw} c_\bw\gb_\bw$ with $\gb_\bb$, $\gb_\bw$ as in \eqref{eq:black cycle}, \eqref{eq:white cycle} and $c_\bb,c_\bw\in\CC$ multiplies all coordinates of the point~\eqref{eq:psi point} with $\exp\Bigl(\sum\limits_{\bb\in\ospb} c_\bb+\sum_{\bw\in\ospw} c_\bw\Bigr)$. Such a replacement therefore shows no effect when we arrive downstairs in projective space. Note that this observation agrees with~\eqref{eq:homology cycle} and $\ho^1\big(\gGLD,\CC^*\big) = \ho^1\big(\gGLD,\ZZ\big)\otimes_\ZZ\CC^*$.

\medskip

\pph{} $\!\!\!$The description \eqref{eq:binomials} of the kernel of the homomorphism \eqref{eq:weight ring} implies that $\Spec(\ZZ[\ocW])$ can be identified with the closed subscheme of the affine space $\aA^{|\ocM|}=\Spec(\ZZ[X_\sm\,|\,\sm\in\ocM])$ given by the equations~\eqref{eq:binomial eqs}. The function $\sum\limits_{\sm\in\ocM} X_\sm$ on $\aA^{|\ocM|}$ restricts to a function on $\Spec(\ZZ[\ocW])$. At the point \eqref{eq:psi point} this function has the value
\begin{gather*}
\sum_{\sm\in\ocM}\exp\bigg(\sum_{e\in\ocE}\sm(e)\psi(e)\bigg) .
\end{gather*}

\subsection{The geometry of a weight realization}\label{sec:geo weight real}

\pph{} Let $(\nu_1,\nu_2,\nu_3)$ be a weight realization. Set $\V=(\nu_1-\nu_3, \nu_2-\nu_3)$ and $\theta=\frac{1}{\deg\nu_3}\nu_3$. For $e\in\ocE$ consider the associated quadrangle. Proposition~\ref{prop:quadrangle} gives values in $\RR^2$ for the vectors from $s(\widetilde{e})$ to $b(\widetilde{e})$, to $w(\widetilde{e})$ and to $t(\widetilde{e})$. We identify $\RR^2$ with $\CC$ and obtain three complex numbers $q_{sb}(e)$, $q_{sw}(e)$, $\V(e)$. We set
\begin{gather}\label{eq:sfq}
\sfq_{\bb}(e)=q_{sb}(e)-\breuk{1}{2}\V(e) ,\qquad \sfq_{\bw}(e)=q_{sw}(e)-\breuk{1}{2}\V(e) .
\end{gather}
The complex numbers $\sfq_{\bb}(e)$ and $\sfq_{\bw}(e)$ are the vectors from the midpoint of the $s(\widetilde{e})t(\widetilde{e})$-diagonal to $b(\widetilde{e})$ and $w(\widetilde{e})$, respectively; see Fig.~\ref{fig:quadrangle2}.

\begin{figure}[t]\centering
\includegraphics{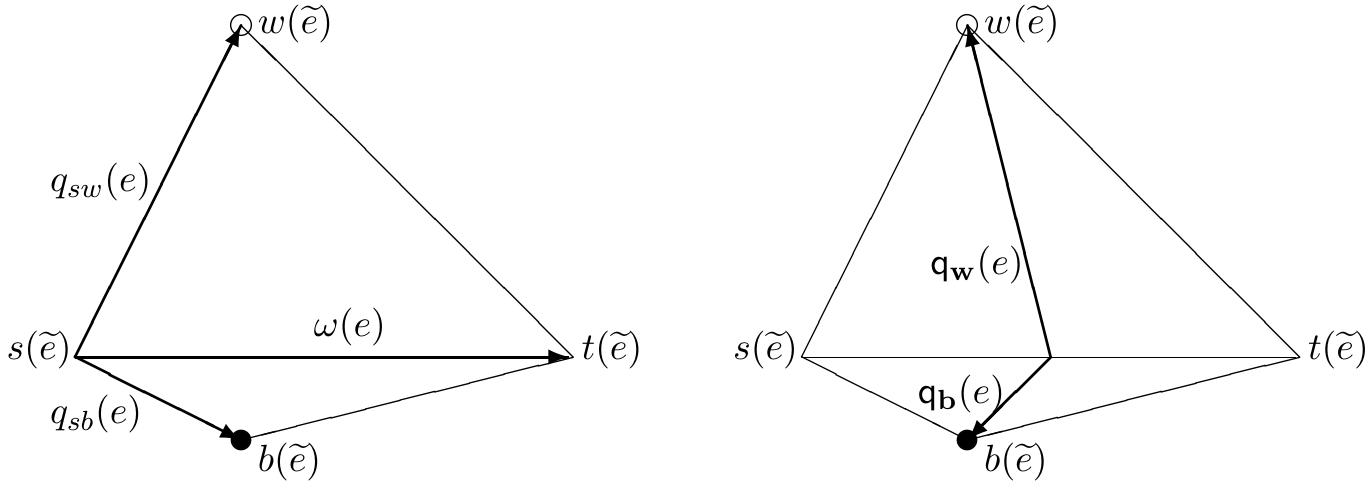}
\caption{Quadrangle for edge $e$.}\label{fig:quadrangle2}
\end{figure}

\begin{Proposition}\label{prop:qbw}
We have for every perfect matching $\sm'$
\begin{gather}\label{eq:mqb}
\sum_{e\in\ocE}\sm'(e)\sfq_{\bb}(e)=(\theta^t-\sm'^t)\cdot\big({-}\breuk{1}{2}\II+\rho_{\sm,1}\big)\cdot\V,\\
\label{eq:mqw}
\sum_{e\in\ocE}\sm'(e)\sfq_{\bw}(e)= (\theta^t-\sm'^t)\cdot\big({-}\breuk{1}{2}\II+\rho_{\sm,0}\big)\cdot\V.
\end{gather}
Here we view $\sm'$ as a column vector of $0$'s and $1$'s and $\V$ as a column vector of complex numbers.\footnote{The perfect matching $\sm$ appearing on the right-hand side has been fixed for the computations in Section~\ref{sec:draw GD}. Since the left-hand side does not involve $\sm$ the results do in fact not depend on the choice of $\sm$.}
\end{Proposition}

\begin{proof} Equations \eqref{eq:bigB} and \eqref{eq:bigW} yield
\begin{gather}\label{eq:mB}
\sm'^t\cdot \mathbf{B}_{\sm,\theta}=\sum_{\bb\in\ospb}\theta^t\cdot\diag(\gb_\bb)\cdot\big({-}\breuk{1}{2}\II+\rho_{\sm,1}\big) = \theta^t\cdot\big({-}\breuk{1}{2}\II+\rho_{\sm,1}\big), \\
\label{eq:mW}
\sm'^t\cdot \mathbf{W}_{\sm,\theta}=\sum_{\bw\in\ospb}\theta^t\cdot\diag(\gb_\bw)\cdot\big({-}\breuk{1}{2}\II+\rho_{\sm,0}\big) = \theta^t\cdot\big({-}\breuk{1}{2}\II+\rho_{\sm,0}\big).
\end{gather}
The results \eqref{eq:mqb} and \eqref{eq:mqw} follow from \eqref{eq:sfq}, \eqref{eq:mB}, \eqref{eq:mW} and Proposition~\ref{prop:quadrangle}.
\end{proof}

\pph{} Subtracting \eqref{eq:mqb} from \eqref{eq:mqw} we get
\begin{gather*}
\sum_{e\in\ocE}\sm'(e)\bigl(\sfq_{\bw}(e)-\sfq_{\bb}(e)\bigr) = \big(\theta^t-\sm'^t\big)\cdot(\rho_{\sm,0}-\rho_{\sm,1})\cdot\V.
\end{gather*}
From \eqref{eq:equivalence matchings} and \eqref{eq:form m 7} we see that if $\sm''\sim\sm'$, then
\begin{gather*}
(\theta^t-\sm''^t)\cdot(\rho_{\sm,0}-\rho_{\sm,1})\cdot\V = \big(\theta^t-\sm'^t\big)\cdot(\rho_{\sm,0}-\rho_{\sm,1})\cdot\V .
\end{gather*}
This means that the map
\begin{gather*}
\ocM\longrightarrow\RR^2 ,\qquad \sm'\mapsto\sum_{e\in\ocE}\sm'(e)\bigl(\sfq_{\bw}(e)-\sfq_{\bb}(e)\bigr)
=\big(\theta^t-\sm'^t\big)\cdot(\rho_{\sm,0}-\rho_{\sm,1})\cdot\V
\end{gather*}
induces an embedding of $\ocA = \ocM/\sim$ and the Newton polygon $\conv(\ocA)$ into $\RR^2$.

\begin{Example}\label{exa:newton polygons}The above method yields for the superpotentials on the Zhegalkin zebra functions $\cF_2$, $\cF_3$, $\cF_4$, $\cF_6$ and $\cF$ in Figs.~\ref{fig:classic1}, \ref{fig:classic2}, \ref{fig:model12b1} and Example~\ref{exa:model12b2} the Newton polygons in Fig.~\ref{fig:Newton polygons}, where we have also indicated the sizes of the fibres of $\ocM\rightarrow\ocA$ if $>1$.
\end{Example}

\begin{figure}[t]\centering
\includegraphics{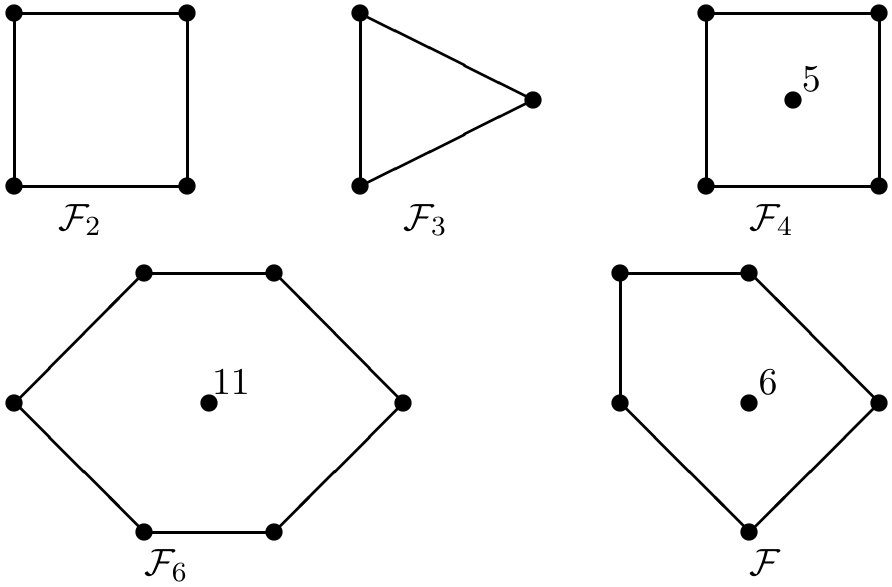}
\caption{Newton polygons for the superpotentials on the Zhegalkin zebra functions $\cF_2$, $\cF_3$, $\cF_4$, $\cF_6$ and $\cF$ in Figs.~\ref{fig:classic1}, \ref{fig:classic2}, \ref{fig:model12b1} and Example \ref{exa:model12b2}.}\label{fig:Newton polygons}
\end{figure}

\pph{} For a map $\psi\colon \ocE\rightarrow\CC$ and $z\in\CC$ we can construct the maps
\begin{gather*}
\psi+z \sfq_\bb ,\qquad \psi+z \sfq_\bw\colon \ \ocE\rightarrow\CC.
\end{gather*}
With the constructions in Section~\ref{pph:points 2} we thus obtain two (parametrized) curves in \linebreak $\CC\Spec(\ZZ[\ocW])$ through the point given by $\psi$:
\begin{gather}\label{eq:Spsi}
 {\cSb}_\psi ,{\cSw}_\psi\colon \ \CC\longrightarrow \CC\Spec(\ZZ[\ocW]) ,\\
\label{eq:Spsib}
 {\cSb}_\psi(z)=
\bigg(\exp\bigg(\sum_{e\in\ocE}\sm'(e)\bigl(\psi(e)+z \sfq_\bb(e)\bigr)\bigg)\bigg)_{\sm'\in\ocM},\\
\label{eq:Spsiw}
{\cSw}_\psi(z)=\bigg(\exp\bigg(\sum_{e\in\ocE}\sm'(e)\bigl(\psi(e)+z \sfq_\bw(e)\bigr)\bigg)\bigg)_{\sm'\in\ocM}.
\end{gather}
It follows from \eqref{eq:poisb}, \eqref{eq:poisw}, \eqref{eq:mqb} and \eqref{eq:mqw} that for every $\sm'\in\ocM$
\begin{gather}
\frac{\partial}{\partial z} \exp \bigg(\sum_{e\in\ocE}\sm'(e)\bigl(\psi(e)+z \sfq_\bb(e)\bigr)\bigg)\nonumber\\
\qquad{} =\breuk{1}{2} \geb\big(\theta^t - \sm'^t,\V\big) \exp\bigg(\sum_{e\in\ocE}\sm'(e)\bigl(\psi(e)+z \sfq_\bb(e)\bigr)\bigg), \label{eq:dzb}\\
\frac{\partial}{\partial z} \exp\bigg(\sum_{e\in\ocE}\sm'(e)\bigl(\psi(e)+z \sfq_\bw(e)\bigr)\bigg) \nonumber\\
\qquad{} = \breuk{1}{2} \gew\big(\theta^t - \sm'^t,\V\big) \exp\bigg(\sum_{e\in\ocE}\sm'(e)\bigl(\psi(e)+z \sfq_\bw(e)\bigr)\bigg) .\label{eq:dzw}
\end{gather}

\pph{}\label{pph:vector fields} Since the coordinates in \eqref{eq:Spsib} and \eqref{eq:Spsiw} are non-zero, these formulas also define two curves
\begin{gather*}
\ol{\cSb}_\psi ,\ol{\cSw}_\psi\colon \ \CC\longrightarrow \ho^1\big(\gGLD,\CC^*\big) \subset \CC\Proj(\ZZ[\ocW]).
\end{gather*}
The symplectic forms $\geb$ and $\gew$ on $\ho^1\big(\gGLD,\CC^*\big)$ and the realization $\V$ yield two vector fields $\geb(\textrm{-},\V)$ and $\gew(\textrm{-},\V)$. The formulas \eqref{eq:dzb}, resp.~\eqref{eq:dzw}, show that the curves $\ol{\cSb}_\psi$, resp.~$\ol{\cSw}_\psi$, are integral curves for these vector fields.

\section{The Jacobi algebra}\label{sec:Jacobi}

Let $\cF$ be a convex Zhegalkin zebra function and $\gL$ a sublattice of $\auto(\cF)$, such that the superpotential $\SFL$ is dimer complete and such that a weight realization of the superpotential exists; see Definitions~\ref{def:zzf}, \ref{def:superpotential}, \ref{def:complete} and~\ref{def:weight realization}. So there is a tiling of the torus $\RR^2/\gL$ by convex black and white polygons. The vertices and edges of these polygons constitute a quiver (=~directed graph)~$\gGL$. In this section we study an algebra which is naturally associated with the embedding of the quiver~$\gGL$ into the torus~$\RR^2/\gL$.

\subsection{General theory of the Jacobi algebra and master space}\label{sec:jacmaster general}

\pph{}\label{sec:jacmaster1}
\begin{Definition}[cf.~\cite{AFM,Bo,Br,FHHZ}]\label{def:Jacobi}\quad
\begin{enumerate}\itemsep=0pt
\item[(i)] The path algebra $\ZZ[\pad(\gGL)]$ of the quiver $\gGL$ is the free abelian group on the set of all paths on~$\gGL$. The product $\bp \bp'$ is the concatenation of $\bp$ and $\bp'$ if $t(\bp)=s(\bp')$ and is~$0$ otherwise. The constant paths supported on the vertices of the quiver are idempotent elements in $\ZZ[\pad(\gGL)]$.

\item[(ii)] The \emph{Jacobi algebra} of the superpotential $\SF_{\gL}$ is the algebra
\begin{gather}\label{eq:Jacobi}
\Jac(\SFL) = \raisebox{.5ex}{$\ZZ[\pad(\gGL)]$}/\raisebox{-.5ex}{$\langle \sD^\circ(e) \,|\, e\in\ocE\rangle$} ,
\end{gather}
where $\langle \sD^\circ(e) \,|\, e\in\ocE\rangle$ is the two sided ideal generated by the elements
\begin{gather}\label{eq:jerel}
\sD^\circ(e) = \prod_{e'\neq e\colon w(e')=w(e)}^{\circlearrowleft}e' - \prod_{e'\neq e\colon b(e')=b(e)}^{\circlearrowright}e' .
\end{gather}

\item[(iii)] The \emph{master space}\footnote{In \cite{AFM,FHHZ} the master space is denoted as $\cF^\flat$.} of $\SFL$ is the scheme $\Spec(\sRf(\SFL))$
with
\begin{gather}\label{eq:master}
\sRf(\SFL) = \raisebox{.5ex}{$\ZZ[X_e\,|\,e\in\ocE]$}/\raisebox{-.5ex}{$\langle \sDf(e) | e\in\ocE\rangle$} ;
\end{gather}
here $\langle \sDf(e) \,| \, e\in\ocE\rangle$ is the ideal generated by the elements
\begin{gather}\label{eq:merel}
\sDf(e) = \prod_{e'\neq e\colon w(e')=w(e)}X_{e'} - \prod_{e'\neq e\colon b(e')=b(e)}X_{e'} .
\end{gather}
\end{enumerate}
\end{Definition}

\begin{Remark}\label{rem:master} Note that $\sDf(e) = \frac{d \mathsf{F}}{dX_e}$ with{\samepage
\begin{gather*}
\mathsf{F}=\sum_{\bw\in\ospw}\prod_{e\in\ocE\colon w(e)=\bw}X_e - \sum_{\bb\in\ospb}\prod_{e\in\ocE\colon b(e)=\bb}X_e .
\end{gather*}
So, $\langle \sDf(e) \,| \, e\in\ocE\rangle$ is the \textit{Jacobi ideal} of the polynomial $\mathsf{F}$.}

The monomials in $\mathsf{F}$ correspond 1-1 with the cycles of the permutations $\gs_0$ and $\gs_1$ with neglect of the cyclic structure. Compare this with Remark \ref{rem:physics potential} and the analogies \eqref{eq:Jacobi}/\eqref{eq:master} and \eqref{eq:jerel}/\eqref{eq:merel}.
\end{Remark}

\pph{} The semi-group dual to $\ocW$ is $\WLD = \Hom(\ocW, \ZZ_{\geq0})$. Evaluation of maps $\ocE\rightarrow\ZZ_{\geq0}$ induces a map
\begin{gather}\label{eq:weight path dual}
\ocE\longrightarrow \WLD , \qquad e\mapsto \bigl(\nu\mapsto\nu(e)\bigr).
\end{gather}
Let $\matel$ denote the ring of matrices with rows and columns indexed by the elements of $\ospv$ and let $\ZZ\big[\WLD\big]$ denote the semi-group ring of $\WLD$. The map \eqref{eq:weight path dual} can then be upgraded to an algebra homomorphism
\begin{gather*}
\widetilde{\Phi}\colon \ \ZZ[\pad(\gGL)] \longrightarrow \matel\big(\ZZ\big[\WLD\big]\big),
\end{gather*}
such that $\widetilde{\Phi}(\bp)$ is the matrix with all entries $0$ except for the $(s(\bp),t(\bp))$-entry, which is $\bp$ viewed as an element of $\WLD$ through \eqref{eq:weight path dual}; i.e.,
\begin{gather*}
\bp=(e_1,\ldots,e_k) \mapsto \bigg(\nu\mapsto\nu(\bp)=\sum_{j=1}^k \nu(e_j)\bigg).
\end{gather*}
It is clear from \eqref{eq:weight condition} and \eqref{eq:jerel} that $\widetilde{\Phi}$ induces an algebra homomorphism
\begin{gather}\label{eq:jacobi weight duality map}
\Phi\colon \ \Jac(\SFL) \longrightarrow \matel\big(\ZZ\big[\WLD\big]\big).
\end{gather}

\begin{Definition}\label{def:tautological representation} We call the above homomorphism $\Phi$ the \textit{tautological representation of the Jacobi algebra}.
\end{Definition}

\pph{}\label{sec:jacmaster2} It follows from \eqref{eq:weight condition} and \eqref{eq:merel} that the ring homomorphism
\begin{gather*}
\ZZ[X_e\,|\,e\in\ocE]\longrightarrow\ZZ\big[\WLD\big] ,\qquad X_e\mapsto (\nu\mapsto\nu(e))
\end{gather*}
induces a ring homomorphism
\begin{gather}\label{eq:irrmaster1}
\sRf(\SFL)\longrightarrow\ZZ\big[\WLD\big]
\end{gather}
and, hence, a morphism of schemes
\begin{gather}\label{eq:irrmaster2}
\Spec\big(\ZZ\big[\WLD\big]\big)\longrightarrow\Spec\big(\sRf(\SFL)\big) .
\end{gather}

\begin{Proposition}\label{prop:irrmaster} The image of the morphism \eqref{eq:irrmaster2} is an irreducible closed subscheme
of the Master Space $\Spec\big(\sRf(\SFL)\big)$.\footnote{In \cite{AFM,FHHZ} the master space is denoted by $\cF^\flat$ and the irreducible component by
${}^{\textrm{Irr}}\!\cF^\flat$.}
\end{Proposition}
\begin{proof} Since $\ocW = \ZZ_{\geq0}\ocM$ by \eqref{eq:weight span} an element $\ga$ of $\WLD$ is completely determined by its values $\ga(\sm)$ for $\sm\in\ocM$. So there is an injective ring homomorphism
\begin{gather}\label{eq:WDinv1}
\ZZ\big[\WLD\big]\longrightarrow\ZZ[U_\sm\,|\,\sm\in\ocM] ,\qquad \ga\mapsto\prod_{\sm\in\ocM}U_{\sm}^{\ga(\sm)} .
\end{gather}
It follows that the ring $\ZZ\big[\WLD\big]$ has no zero-divisors and that the kernel of the ring homomorphism~\eqref{eq:irrmaster1} is a prime ideal.
\end{proof}

\pph{} With $\MM$ as in equation~\eqref{eq:matchrel} we have
\begin{gather*}
\WLD=\bigg\{\ga\in\ZZ^{\ocM}\,|\,\forall\, \mu\in\MM \colon \sum_{\sm\in\ocM}\ga(\sm)\mu(\sm)=0\bigg\}.
\end{gather*}
From this we see that \eqref{eq:WDinv1} identifies $\CC\big[\WLD\big]$ with the sub-ring of $\CC[U_\sm\,|\,\sm\in\ocM]$ consisting of those polynomials which are invariant under the (obvious) action of the torus $\MM\otimes\CC^*$:
\begin{gather}\label{eq:WDinv2}
\CC\big[\WLD\big]=\CC[U_\sm\,|\,\sm\in\ocM]^{\MM\otimes\CC^*}.
\end{gather}
The situation described by \eqref{eq:WDinv2} is in an obvious sense dual to the situation described in \eqref{eq:weight ring}--\eqref{eq:binomial eqs}.

\begin{Remark}\label{rem:master 2}The story in \eqref{eq:master}, \eqref{eq:merel}, \eqref{eq:matchrel}, \eqref{eq:WDinv2} is well-known. It differs from the discussion of the master space and its irreducible component in \cite{AFM, FHHZ} only in terminology and style and in that we have highlighted the role of the weight functions.
\end{Remark}

\pph{} By definition the \textit{center of the Jacobi algebra} is
\begin{gather*}
\sZ(\Jac(\SFL)) = \{\pi\in\Jac(\SFL)\,|\,\forall\, e\in\ocE\colon \pi e=e\pi \}.
\end{gather*}
Applying $\Phi$ \eqref{eq:jacobi weight duality map} to an element $\pi$ in $\sZ(\Jac(\SFL))$ yields the matrix equations
\begin{gather*}
\forall\, e\in\ocE\colon \ \Phi(\pi)\Phi(e)=\Phi(e)\Phi(\pi).
\end{gather*}
For $\nu\in\ocW$ ``evaluation at $\nu $'' defines a homomorphism of semi-groups $\WLD\rightarrow\ZZ_{\geq0}$ and, hence, a homomorphism of rings $\ZZ\big[\WLD\big]\rightarrow\ZZ[u]$. By combining this homomorphism with $\Phi$ we obtain an algebra homomorphism
\begin{gather*}
\Phi_\nu\colon \ \Jac(\SFL) \longrightarrow \matel(\ZZ[u]),
\end{gather*}
such that for every $e\in\ocE$ the only non-zero entry of the matrix $\Phi_\nu(e)$ is $u^{\nu(e)}$ in position $(s(e),t(e))$. The matrix equations
\begin{gather*}
\forall\, e\in\ocE\colon \ \Phi_\nu(\pi)\Phi_\nu(e)=\Phi_\nu(e)\Phi_\nu(\pi)
\end{gather*}
then imply that there is an element $c_\pi\in\ZZ\big[\WLD\big]$ such that
\begin{gather*}
\Phi(\pi) = c_\pi \II ,\qquad\text{i.e.},\qquad \forall\,\nu\in\ocW\colon \ \Phi_\nu(\pi) = c_\pi(\nu)\II .
\end{gather*}
Now let $\nu, \nu'\in\ocW$ be such that $\nu\sim\nu'$ (see \eqref{eq:equivalence matchings}), say
\begin{gather}\label{eq:nunu1}
\nu-\nu'=\sum_{\bv\in\ospv} r_\bv\ga_\bv\qquad\textrm{with}\quad r_\bv\in\ZZ ,\quad\forall\, \bv\in\ospv.
\end{gather}
Then we have for all $e\in\ocE$:
\begin{gather}\label{eq:nunu2}
\Phi_{\nu'}(e) = \diag\bigl(\big(u^{r_\bv}\big)_{\bv\in\ospv}\bigr)\cdot\Phi_{\nu}(e)\cdot \diag\bigl(\big(u^{-r_\bv}\big)_{\bv\in\ospv}\bigr).
\end{gather}
Consequently we have for $\pi$ in $\sZ(\Jac(\SFL))$:
\begin{gather}\label{eq:nunu3}
\Phi_{\nu'}(\pi) = \diag\bigl(\big(u^{r_\bv}\big)_{\bv\in\ospv}\bigr)\cdot\Phi_{\nu}(\pi)\cdot \diag\bigl(\big(u^{-r_\bv}\big)_{\bv\in\ospv}\bigr)
\end{gather}
and, hence, $c_{\pi}(\nu') = c_{\pi}(\nu)$. This means that $c_{\pi}$ is actually an element of the semi-group ring $\ZZ\big[\oWLD\big]$ of the semi-group $\oWLD$ dual to $\oocW$.

This proves:

\begin{Proposition}\label{prop:center jac}There is an algebra homomorphism
\begin{gather}\label{eq:Phi center}
\sZ(\Jac(\SFL)) \longrightarrow \ZZ\big[\oWLD\big]
\end{gather}
induced by the tautological representation $\Phi$ \eqref{eq:jacobi weight duality map}.
\end{Proposition}

\subsection{Weight realizations and Jacobi algebra}\label{sec:jacmasterweight}

\pph{} For a weight realization $(\nu_1,\nu_2,\nu_3)$ ``evaluation'' at $(\nu_1,\nu_2,\nu_3)$ defines a homomorphism of rings $\ZZ\big[\WLD\big]\rightarrow\ZZ[u_1,u_2,u_3]$. By combining this homomorphism with $\Phi$ we obtain an algebra homomorphism
\begin{gather}\label{eq:jacobi weight nu123}
\Phi_{\nu_1,\nu_2,\nu_3}\colon \ \Jac(\SFL) \longrightarrow \matel(\ZZ[u_1,u_2,u_3]),
\end{gather}
such that for every $e\in\ocE$ the only non-zero entry of the matrix $\Phi_{\nu_1,\nu_2,\nu_3}(e)$ is $u_1^{\nu_1(e)}u_2^{\nu_2(e)}u_3^{\nu_3(e)}$ in position $(s(e),t(e))$. So, in particular
\begin{gather*}
\Phi_{\nu_1,\nu_2,\nu_3}\bigg(\sum_{e\in\ocE} e\bigg) = \mav\big(u_1^{\nu_1}u_2^{\nu_2}u_3^{\nu_3}\big)
\end{gather*}
with matrix $\mav\big(u_1^{\nu_1}u_2^{\nu_2}u_3^{\nu_3}\big)$ as in \eqref{eq:Anu123}.

\medskip

\pph{} Let $(\nu_1,\nu_2,\nu_3)$ and $(\nu'_1,\nu'_2,\nu'_3)$ be weight realizations such that $\nu_1\sim\nu'_1$, $\nu_2\sim\nu'_2$, $\nu_3\sim\nu'_3$; see~\eqref{eq:equivalence matchings}. Then one can show as in \eqref{eq:nunu1}--\eqref{eq:nunu3} that there is a diagonal matrix $D$ such that
\begin{gather*}
\mav\big(u_1^{\nu'_1}u_2^{\nu'_2}u_3^{\nu'_3}\big) = D\cdot\mav\big(u_1^{\nu_1}u_2^{\nu_2}u_3^{\nu_3}\big)\cdot D^{-1}.
\end{gather*}
This means that the representations $\Phi_{\nu_1,\nu_2,\nu_3}$ and $\Phi_{\nu'_1,\nu'_2,\nu'_3}$ are isomorphic.

\medskip

\pph{}\label{pph:faithful representation}%
The matrix $\mav\big(u_1^{\nu_1}u_2^{\nu_2}u_3^{\nu_3}\big)$ contains the complete information on the edge vectors with which one can draw the quiver $\gG$ embedded in $\RR^2$. More precisely, consider a path $\bp=(e_1,\ldots,e_k)$ and its subpaths $\bp_j=(e_1,\ldots,e_j)$ for $j=1,\ldots,k$. The path $\bp_j$ corresponds to the monomial
\begin{gather*}
\prod_{r=1}^j u_1^{\nu_1(e_r)}u_2^{\nu_2(e_r)}u_3^{\nu_3(e_r)}
\end{gather*}
in the $(s(e_1),t(e_j))$-entry of the matrix $\bigl(\mav\big(u_1^{\nu_1}u_2^{\nu_2}u_3^{\nu_3}\big)\bigr)^j$. In this way one sees the actual path~$\bp$ as it runs through the end points of the subpaths $\bp_j$, $j=1,\ldots,k$. Thus one obtains from $\mav\big(u_1^{\nu_1}u_2^{\nu_2}u_3^{\nu_3}\big)$ the paths on $\gGL$, the boundary cycles of the polygons and the period lattice. Since this is all one needs for \eqref{eq:Jacobi}--\eqref{eq:jerel} we conclude:

\begin{Theorem}\label{thm:faithful}The algebra homomorphisms
\begin{align*} 
\Phi_{\nu_1,\nu_2,\nu_3}\colon \ \Jac(\SFL) &\longrightarrow \matel(\ZZ[u_1,u_2,u_3]),\\
\Phi\colon \ \Jac(\SFL) &\longrightarrow \matel\big(\ZZ\big[\WLD\big]\big),\\
\sZ(\Jac(\SFL)) &\longrightarrow \ZZ\big[\oWLD\big]
\end{align*}
in \eqref{eq:jacobi weight nu123}, \eqref{eq:jacobi weight duality map}, \eqref{eq:Phi center} are injective.
\end{Theorem}

\begin{Remark}\label{rem:cancellative}Theorem \ref{thm:faithful} is closely related to Theorem~3.17 and Definition~3.10 in~\cite{Bo}. So it seems that the quivers with potential coming from Zhegalkin zebra functions are \textit{cancellative} in the sense of \cite[Definition~3.10]{Bo}.

Theorem~3.20 in \cite{Bo} states that $\Jac(\SFL)$ is a non-commutative crepant resolution of the $3$-dimensional Gorenstein singularity $\Spec\bigl(\sZ(\Jac(\SFL))\bigr)$ if the quiver $\gGL$ with superpoten\-tial~$\SFL$ is cancellative.

On the other hand, we do have examples of Zhegalkin zebra functions for which the quiver with potential is not \textit{consistent} in the sense of \cite[Theorems~1.37 and 3.11]{Bo}.
\end{Remark}

\begin{Remark}\label{rem:generating series} The above method of generating paths corresponds to the series expansion
\begin{gather*}
\bigl(\II-\mav\big(u_1^{\nu_1}u_2^{\nu_2}u_3^{\nu_3}\big) \bigr)^{-1} = \sum_{j=0}^\infty \bigl(\mav\big(u_1^{\nu_1}u_2^{\nu_2}u_3^{\nu_3}\big)\bigr)^j.
\end{gather*}
Since every entry of the matrix $\mav\big(u_1^{\nu_1}u_2^{\nu_2}u_3^{\nu_3}\big)$ is divisible by $u_1u_2u_3$ the series on the right-hand side converges in the topology provided by the powers of the principal ideal $u_1u_2u_3\ZZ[u_1,u_2,u_3]$.
\end{Remark}

\section{Symptoms of mirror symmetry}\label{sec:mirror symmetry}

In this section we put the results of Sections \ref{sec:geoweights} and \ref{sec:Jacobi} into the perspective of mirror symmetry. There are evidently two sides to the story with the graphs $\gGLD$ and $\gGL$ on different sides and the S-quad-graph $\cD_{\cF,\gL}$ providing a ``mirror correspondence''. The appearance of the semi-group ring $\ZZ[\ocW]$ of $\ocW$ on one side and the semi-group ring $\ZZ\big[\WLD\big]$ of the dual semi-group $\WLD = \Hom(\ocW, \ZZ_{\geq0})$ on the other side is reminiscent of \textit{mirror symmetry as in the work of Batyrev and Borisov} \cite{Ba,BB}.

\textbf{The $\boldsymbol{\gGLD}$-side:} The semi-group $\ocW$ is to be put on the $\gGLD$-side, because (see Section~\ref{pph:WGLD})
\begin{gather*}
\ho^1\big(\gGLD, \ZZ\big) = \bigl\{\nu'-\nu''\in\ZZ^{\ocE}\,|\,\nu',\nu''\in\ocW, \, \deg\nu'=\deg\nu''\bigr\} .
\end{gather*}
$\gGL$ induces an equivalence relation $\sim$ on $\ocW$ (cf.\ Section~\ref{pph:equivalence}):
\begin{gather*}
\nu\sim\nu' \Longleftrightarrow \exists\, r\colon \ospv\rightarrow\ZZ \ \text{s.t.} \ \forall\, e\in\ocE\colon \nu(e)-\nu'(e)=r(t(e))-r(s(e)).
\end{gather*}
This corresponds to the equivalence relation on $\ho^1\big(\gGLD,\ZZ\big)$ given by the subgroup generated by the maps $\ga_\bv\colon \ocE\rightarrow\ZZ$ defined in \eqref{eq:vertex cycle}. Through equations~\eqref{eq:D graph in torus}--\eqref{eq:dual face} it can be traced back to the inclusion $\gGLD\hookrightarrow\RR^2/\gL$.

The set of $\sim$-equivalence classes in $\ocW$ is denoted by $\oocW$; see \eqref{eq:Wsim}. The diagram in equation~\eqref{eq:specproj} shows the various schemes and their interrelations associated with the (graded) semi-groups $\ocW$ and $\oocW$ and the group $\ho^1\big(\gGLD, \ZZ\big)$.

\begin{Remark}\label{rem:deformation actions} Notice the analogy between formula~\eqref{eq:deform weights} for the deformations of weight realizations and the action of $\CC^*$ on $\Spec(\ZZ[\ocW])$ and that of $\ho^1\big(\gGLD,\CC^*\big)$ on $\Proj(\ZZ[\ocW])$.
\end{Remark}

\textbf{The $\boldsymbol{\gGL}$-side:} The counterpart of $\ocW$ on the $\gGL$-side is the \textit{Jacobi algebra} $\Jac(\SFL)$ of the superpotential $\SFL$. This is the quotient of the path algebra $\ZZ[\pad(\gGL)]$ of the quiver~$\gGL$ by a~two-sided ideal provided by the permutations~$\gs_0$ and~$\gs_1$; see~\eqref{eq:Jacobi} for a precise definition based on~\cite{Bo,Br}. The Jacobi algebra comes with an injective algebra homomorphism, the \textit{tautological representation},
\begin{gather*}
\Phi\colon \ \Jac(\SFL) \longrightarrow \matel\big(\ZZ\big[\WLD\big]\big)
\end{gather*}
into the algebra of $|\ospv|\times|\ospv|$-matrices over the semi-group ring $\ZZ\big[\WLD\big]$ of the semi-group $\WLD$ dual to $\ocW$. It restricts to an injective algebra homomorphism
\begin{gather*}
\sZ(\Jac(\SFL))\longrightarrow\ZZ\big[\oWLD\big]
\end{gather*}
from the center of the Jacobi algebra into the semi-group ring $\ZZ\big[\oWLD\big]$; see Theorem \ref{thm:faithful}.

\textbf{The $\boldsymbol{\cD_{\cF,\gL}}$-correspondence:} A weight realization $(\nu_1,\nu_2,\nu_3)$ gives rise to a tiling of $\RR^2$ by quadrangles and, hence, for every $e\in\ocE$ vectors $\sfq_\bb(e)$ and $\sfq_\bw(e)$ as in Figs.~\ref{fig:quadrangle3} and~\ref{fig:quadrangle2}. The vector $\sfq_\bb(e)-\sfq_\bw(e)$ is the diagonal from $w(e)$ to $b(e)$ in the quadrangle.
\begin{gather*}
\fbox{\parbox[t]{12cm}{\textit{In this way the map $\sfq_\bb-\sfq_\bw\colon \ocE\rightarrow\RR^2$ realizes the duality between the graphs $\gGL$ and $\gGLD$.}}}
\end{gather*}
For a perfect matching $\sm$ we set, with the notations as in \eqref{eq:rho01},
\begin{gather*}
\tau_{\sm,0}=(\II+\vgs_{\sm,0})(\II-\vgs_{\sm,0})^{-1} ,\qquad \tau_{\sm,1}=(\II+\vgs_{\sm,1})(\II-\vgs_{\sm,1})^{-1} .
\end{gather*}
Then $\tau_{\sm,0}$ and $\tau_{\sm,1}$ are unipotent matrices of size $|\ocE|\times|\ocE|$ with entries in $\ZZ_{\geq0}$. They define injective homomorphisms $\sT_{\sm,0}$ and $\sT_{\sm,1}$ of semi-groups (see~\eqref{eq:Wduality maps})
\begin{gather*}
\sT_{\sm,j}\colon \ \ocW\longrightarrow\WLD ,\qquad \bigl(\sT_{\sm,j}(\nu)\bigr)(\nu')=\nu'^t\cdot\tau_{\sm,j}\cdot\nu .
\end{gather*}
Here $\nu,\nu'\in\ocW$ are viewed as column vectors.

Using the vectors $\sfq_\bb(e)$ and $\sfq_\bw(e)$ we define maps $\mathsf{Q}_\bb, \mathsf{Q}_\bw\colon \ocW\longrightarrow\RR^2$,
\begin{gather*}
\mathsf{Q}_\bb(\nu)=\sum_{e\in\ocE}\nu(e)\sfq_\bb(e) ,\qquad \mathsf{Q}_\bw(\nu)=\sum_{e\in\ocE}\nu(e)\sfq_\bw(e) .
\end{gather*}
Formulas \eqref{eq:weight span I} and \eqref{eq:mqb}--\eqref{eq:mqw} with $\V=(\nu_1-\nu_3, \nu_2-\nu_3)$ and $\theta=\frac{1}{\deg\nu_3}\nu_3$ then show that the maps $\mathsf{Q}_\bb$, $\mathsf{Q}_\bw$ can be expressed as linear combinations of
\begin{gather*}
\sT_{\sm,0}(\nu_1), \ \sT_{\sm,0}(\nu_2), \ \sT_{\sm,0}(\nu_3), \ \sT_{\sm,1}(\nu_1), \ \sT_{\sm,1}(\nu_2), \ \sT_{\sm,1}(\nu_3).
\end{gather*}

Every perfect matching $\sm$ yields two matrices $\tau_{\sm,0}$ and $\tau_{\sm,1}$ with entries in $\ZZ_{\geq0}$ and determinant $1$. Products and transposes of such matrices also have entries in $\ZZ_{\geq0}$ and determinant~$1$.
\begin{gather*}
\fbox{\parbox[t]{12cm}{\textit{In this way one obtains lots of maps from $\ocW$ to $\WLD$. It would be nice if these can be used to built a correspondence between the toric geometry of $\Proj(\ZZ[\ocW])$ on the $\gGLD$-side and the non-commutative algebraic geometry of $\Jac(\SFL)$ on the $\gGL$-side.}}}
\end{gather*}
We leave further analysis of this structure for future research.

\section{Practical matters}\label{sec:practical}

In this section I describe some methods for using a computer to draw the tiling associated with a Zhegalkin zebra polynomial $\cF$, compute the superpotential $\SFL$ and check some conditions. Although the ideas work quite generally the exposition here is strongly influenced by my habit of using \textsc{matlab}.

\subsection[How to draw the picture of the tiling of $\cF$]{How to draw the picture of the tiling of $\boldsymbol{\cF}$}\label{subsec:draw}

The defining formula for a \ZZF{} $\cF$ can be rewritten as formula~\eqref{eq:evaluate} (below) with which one can easily draw the picture of the tiling. For the description of formula~\eqref{eq:evaluate} we define the function $\neg\colon \RR\rightarrow\{0,1\}$ by $\neg(r)=1$ if $r=0$ and $\neg(r)=0$ if $r\neq0$, we identify $0,1\in\FF_2$ with $0,1\in\RR$ and we interpret in the matrix operations the matrix entries as elements of $\RR$.

Extract from the defining formula for $\cF$ the $2\times n$-matrix $V$ of which the columns are the used frequency vectors. Put the coordinates of the points at which the function should be evaluated as rows in a $k\times2$-matrix $X$. Compute the matrix $2 X\cdot V$ and apply the function $\lfloor\;\rfloor\bmod 2$ to its entries. In short hand notation this can be summarized as $\lfloor 2 X\cdot V\rfloor \bmod 2$.

Extract from the defining formula for $\cF$ the $n\times m$-matrix $M$ with entries $0, 1$ of which the columns correspond to the monomials in the formula. Note that a monomial evaluates to $1$ if and only if all variables it involves have value $1$. This leads to the formula $\neg\bigl((\neg(\lfloor 2 X\cdot V\rfloor \bmod 2))\cdot M\bigr)$ for evaluating the monomials. In this formula the function $\neg$ is applied to the entries of the matrices.

The next and final step is to take the sum of the columns (or, equivalently, multiply on the right by the column vector $\mathbf{1}$ consisting of $m$ $1$'s) and reduce the result modulo $2$. The result is a vector of $0$'s and $1$'s which gives the value of $\cF$ at the points listed in $X$. Thus the whole evaluation process reads:
\begin{gather}\label{eq:evaluate}
\cF(X) = \big(\big( \neg\big(\big( \neg(\lfloor 2 X\cdot V\rfloor \bmod 2)\big)\cdot M\big)\big)\cdot\mathbf{1}\big)\bmod 2.
\end{gather}

\begin{figure}[t]\centering

\raisebox{25mm}[0pt][0pt]{$\begin{array}{@{}l}
V = (\fv_2, \fv_3, \fv_4, \fv_6, 2\fv_4)\\[1ex]
M = \left(
\begin{matrix}
1&0&0&0&0\\
0&1&0&0&1\\
0&0&1&0&0\\
0&0&0&1&1\\
0&0&0&0&1
\end{matrix}
\right)\end{array}$} \qquad \includegraphics[width=5.5cm]{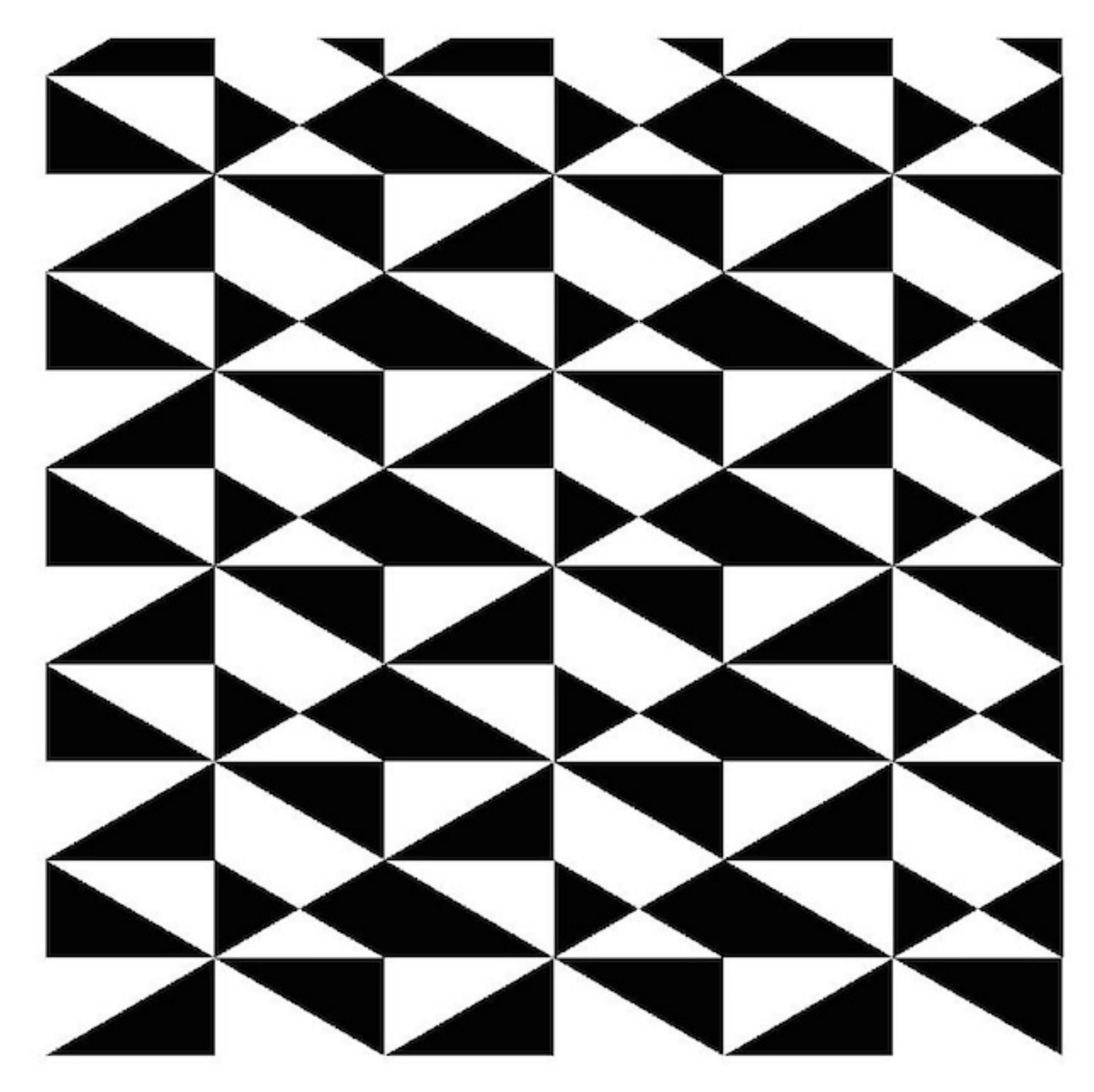}

\caption{Matrices $V$, $M$ and picture for $\cF=\ze{21}+\ze{31}+\ze{41}+\ze{61}+\ze{31}\ze{42}\ze{61}$.}\label{fig:model12b1}
\end{figure}

\subsection[How to compute the superpotential $\SFL$ and the realization $\V_\cF$]{How to compute the superpotential $\boldsymbol{\SFL}$ and the realization $\boldsymbol{\V_\cF}$}\label{subsec:compot}

We now explain how one can compute the superpotential $\SFL=(\ocE,\ogs_0,\ogs_1)$ from the defining formula of the \ZZF{} $\cF$ and the lattice $\gL\subset\auto(\cF)$.

\medskip

\pph{}\label{step 1} The first step is to multiply the vectors $\fv_1,\ldots,\fv_6$ by the diagonal matrix $\operatorname{diag}(2, 2/\sqrt{3})$. This results in the new basic frequency vectors
\begin{gather*}
\tilde{\fv}_1=\left( \begin{matrix}-3\\ \hphantom{-}1\end{matrix}\right),\quad
\tilde{\fv}_2=\left(\begin{matrix}-1\\ \hphantom{-}1\end{matrix}\right),\quad
\tilde{\fv}_3=\left(\begin{matrix}0\\ 2\end{matrix}\right),\quad
\tilde{\fv}_4=\left(\begin{matrix}1\\ 1\end{matrix}\right),\quad
\tilde{\fv}_5=\left(\begin{matrix}3\\ 1\end{matrix}\right),\quad
\tilde{\fv}_6=\left(\begin{matrix}2\\ 0\end{matrix}\right),
\end{gather*}
with which we reinterpret the frequency vectors in the defining formula for $\cF$. This clearly does not change the combinatorial structure of the picture, but it allows to do most computations with integer arithmetic.

As in Section~\ref{subsec:draw} we let the matrix $V$ be such that its columns are the frequency vectors of~$\cF$. The edges in the picture lie on lines with equation $2\bx\cdot\fv = m$ with $m\in\ZZ$ and $\fv$ a column of $V$. The vertices are intersection points of two such lines $2\bx\cdot\fv = m$ and $2\bx\cdot\fv' = m'$ with linearly independent $\fv$ and $\fv'$. The coordinates of the intersection point are then rational numbers with denominators dividing the number $2|\det(\fv,\fv')|/|\mathrm{gcd}(\textrm{entries of }\fv,\fv')|$. Let $K$ denote the least common multiple of these numbers for $(\fv,\fv')$ running over all pairs of linearly independent frequency vectors of~$\cF$.

Fix a sufficiently large\footnote{$N$ should at least be so large that the periodicity lattice which we want to implement later has two basis vectors in $\RR^2$ with non-negative coordinates $\leq \frac{1}{3}N$.} positive integer~$N$. Let $X$ be the $N^2\times2$ -matrix with set of rows $\{(n,m)\in\ZZ^2 \,|\, 0\leq n,m< N\}$ such that row $(n,m)$ is above row $(n',m')$ if $n+m\sqrt{2}<n'+m'\sqrt{2}$.

We want to find among the rows of $\frac{1}{K}X$ those which are intersection points of two lines $2\bx\cdot\fv = m$ and $2\bx\cdot\fv' = m'$ with linearly independent $\fv$ and $\fv'$. First we determine which entries of the matrix $2X\cdot V$ are divisible by $K$; with the notation of Section~\ref{subsec:draw} this means the entries~$1$ in the matrix $\neg((2X\cdot V)\bmod K)$. The intersection points correspond to the rows with a $1$ for at least two linearly independent frequency vectors. Let $X^*$ denote the submatrix of $X$ given by this selection of rows.
Correspondingly we have the two matrices
\begin{gather*}
2X^*\cdot V\qquad\textrm{and}\qquad \neg((2X^*\cdot V)\bmod K) .
\end{gather*}
Looking at these two matrices column by column one easily determines what are the relevant lines and how the points in $X^*$ divide these lines into closed intervals with non-overlapping interiors. For this the initial ordering of the elements of $X$ is very useful. Since it can happen that the same interval is produced from two different columns we remove the duplicates retaining for each interval exactly one copy. We list the intervals thus found by giving for each the two endpoints.

The above calculations were done with integer arithmetic. In the next steps we have to work in $\RR^2$ and must therefore divide for all intervals in our list the coordinates of the endpoints by~$K$. For each interval in the list, say~$I$, take on both sides of the interval a point close to the midpoint and evaluate $\cF$ at these two points using formula~\eqref{eq:evaluate}. Remove $I$ from the list if $\cF$ has at these two points the same value. What is left is a list of intervals separating black and white regions. For each of these intervals we call one endpoint the source and the other the target, so that going along the interval from source to target the black region is on the right.

We make a new list with for each interval $I$ besides the endpoints $s(I)$ and $t(I)$ also the midpoint $m(I)=\frac{1}{2}(s(I)+t(I))$ and the vector $\vec{I}=t(I)-s(I)$:
\begin{gather}\label{eq:edgelist}
\bigl\{ m(I) , s(I) , t(I) , \vec{I} \,\bigr\}_{I} .
\end{gather}

\pph{} \label{pph:auto} We use the list \eqref{eq:edgelist} to compute the group $\auto(\cF)$ of translations which leave the function $\cF$ invariant. Let $I_1$ be the first item in this list. Determine all $I$ with $\vec{I}=\vec{I_1}$ and compute for each of these the vector $T(I)=m(I)-m(I_1)$. In order to check which of these vectors $T(I)$ leave $\cF$ invariant we take at random a point $\bx$ in the unit square and compute $\cF(\bx+T(I))$ and $\cF(\bx)$. We remove $I$ if $\cF(\bx+T(I))\neq\cF(\bx)$. The remaining vectors $T(I)$ are then tested against a new randomly chosen $\bx$. The vectors which are left after repeating this procedure a good number of times generate $\auto(\cF)$. From these generators we choose a basis for $\auto(\cF)$.

\medskip

\pph{}\label{pph:permutations} Having a basis for $\auto(\cF)$ we can specify the desired periodicity lattice $\gL$ by an integer $(2\times2)$-matrix with non-zero determinant. In order to make the reduction modulo $\gL$ we fix a~basis for $\gL$ and write~$m(I)$,~$s(I)$,~$t(I)$ and~$\vec{I}$ in coordinates with respect to this basis. Reduction modulo $\gL$ is achieved by taking the fractional parts of the coordinates of~$m(I)$,~$s(I)$,~$t(I)$ leaving~$\vec{I}$ unchanged. This results in the list
\begin{gather}\label{eq:rededgelist}
\bigl\{ m(I)-\lfloor m(I)\rfloor , s(I)-\lfloor s(I)\rfloor , t(I)-\lfloor t(I)\rfloor , \vec{I}\, \bigr\}_{I} .
\end{gather}
The vectors in \eqref{eq:rededgelist} are given by their coordinates with respect to the chosen basis of $\gL$. Converting this back to the original coordinates on $\RR^2$ and multiplying by $2K$ turns \eqref{eq:rededgelist} into a list of quadruples of elements of $\ZZ^2$. The first three elements in these quadruples have non-negative coordinates $<2K$ and can be made into integers using the injective map $\{0,\ldots,2K-1\}\times\{0,\ldots,2K-1\}\rightarrow\NN$, $(a,b)\mapsto a+2bK$. These integers can be used as labels to identify the midpoint, source and target of the edge.

The list which thus results from \eqref{eq:rededgelist} contains many duplicates, which we remove. What remains is a list of labeled edges $e$ with labeled source $s(e)$ and target $t(e)$ and the edge vector~$\ve(e)$. It can still happen that this list contains edges $e$ for which there is only one edge $e'$ with $s(e')=t(e)$, which case $e$ and $e'$ must be fused. This will be taken care of in the final part of the next step.

\begin{figure}[t]\centering
\includegraphics{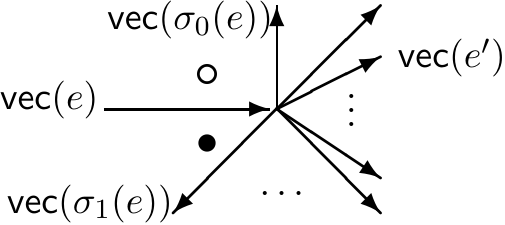}
\caption{Situation at the target of edge $e$.}\label{fig:vertex star}
\end{figure}

Let $e$ be an edge in the list. Take all edges $e'$ in the list satisfying $s(e')=t(e)$. Using the vectors $\ve(e)$ and $\ve(e')$ we define for every such $e'$ the number
\begin{gather*}
C(e') = \left( 1 - \frac{\ve(e)\cdot\ve(e')}{|\ve(e)| |\ve(e')|}\right) \operatorname{sign}\bigl(\det(\ve(e),\ve(e'))\bigr) .
\end{gather*}
Finally we define the edges
$\ogs_1(e)$ and $\ogs_0(e)$ by $s(\ogs_1(e)) = s(\ogs_0(e)) = t(e)$ and
\begin{gather*}
C(\ogs_1(e))\leq C(e')\leq C(\ogs_0(e))\qquad\text{for all $e'$ with $s(e')=t(e)$.}
\end{gather*}
This yields the permutations $\ogs_0$ and $\ogs_1$ of the set of edges in our list.

Edges $e$ for which there is only one edge $e'$ with $s(e')=t(e)$ can be recognized as those for which $\ogs_1(e) = \ogs_0(e)$. Now repeat the following procedure as long as there are edges with this undesirable property. Take the first edge $e$ for which $\ogs_1(e) = \ogs_0(e)$. Add a new edge $e^*$ to the list, setting $\text{label}(e^*)=\text{label}(e)$, $s(e^*)=s(e)$, $t(e^*)=t(\ogs_0(e))$, $\ogs_1(e^*) = \ogs_1^2(e)$, $\ogs_0(e^*) = \ogs_0^2(e)$, $\ve(e^*)=\ve(e)+\ve(\ogs_0(e))$ and remove $e$ and $\ogs_0(e)$ from the list.

\medskip

\pph{}\label{pph:real F} After all this we have obtained the list of elements of $\ocE$ and for every $e\in\ocE$ the edge vector $\ve(e)$ as well as $\ogs_0(e)$ and $\ogs_1(e)$. The superpotential is then $\SFL=(\ocE,\gs_0,\gs_1)$ and $\V_{\cF}(e)=\ve(e)$ gives the realization $\V_{\cF}$ of $\SFL$.

\begin{Example}\label{exa:model12b2}For $\cF$ as in Fig.~\ref{fig:model12b1} we find $\auto(\cF)=\frac{1}{3}\fv_3\ZZ\oplus\fv_6\ZZ$. For $\gL=\auto(\cF)$ the superpotential $\SFL=(\ocE,\gs_0,\gs_1)$ is given by
\begin{gather*}
\gs_0=(1,3,13)(4,6,2)(7,10,11)(9,5,8,12) ,\qquad \gs_1=(1,5,2)(4,3,11,12)(7,6,8)(9,10,13) .
\end{gather*}
One can recover the labeling of the edges by matching $\gs_0$ and $\gs_1$ with the picture in Fig.~\ref{fig:model12b1}.
\end{Example}

\subsection{How to compute the lattice and the quadrangles of a weight realization}\label{subsec:comlatquad}

\pph{}\label{pph:latweight} From Sections~\ref{pph:auto} and \ref{pph:real F} we have obtained a list of edge vectors $\{\V_{\cF}(e)\}_{e\in\ocE}$ for the realization $\V_{\cF}$ and a basis $\gl_1,\gl_2$ for the lattice $\gL_{\V_{\cF}}$. This implies that there are maps $\sn_1,\sn_2\colon \ocE\rightarrow\ZZ$ such that for $j=1,2$:
\begin{gather}*\label{eq:lat F1}
\gl_j=\sum_{e\in\ocE} \sn_j(e)\V_{\cF}(e) ,\nonumber\\
\label{eq:lat F2}
\sum_{e\in\ocE} \sn_j(e) e\quad\textrm{is the homology class of a closed path on} \ \gGL ;
\end{gather}
see \eqref{eq:loop class}. A basis $\gl_{\V,1},\gl_{\V,2}$ of the lattice $\gL_\V$ for any other realization $\V$ of the superpoten\-tial~$\SFL$ is then given by
\begin{gather*}
\gl_{\V,j}=\sum_{e\in\ocE} \sn_j(e)\V(e) \qquad \textrm{for} \quad j=1,2.
\end{gather*}
The following method avoids the explicit computation of $\sn_1$, $\sn_2$ (which are not uniquely determined anyway). Fix $\bv_0\in\ospv$ and define the matrices $\EV_\cF$ and $\EV_\V$ with $|\ocE|$ rows and $|\ospv|+1$ columns by: the first two columns of $\EV_\cF$ are $\V_\cF$ and the first two columns of $\EV_\V$ are $\V$; the last $|\ospv|-1$ columns of both $\EV_\cF$ and $\EV_\V$ are $\ga_\bv$ with $\bv\in\ospv$, $\bv\neq\bv_0$; here $\ga_\bv$ is as in~\eqref{eq:vertex cycle}. Schematically:
\begin{gather*}
\EV_\cF=[ \V_\cF | A ] ,\qquad \EV_\V=[ \V | A ] .
\end{gather*}
Viewing $\sn_1$ and $\sn_2$ as row vectors we have for $j=1,2$
\begin{gather*}
\sn_j\cdot\EV_\cF=[ \gl_j | \nul ] ,\qquad \sn_j\cdot\EV_\V=[ \gl_{\V,j} | \nul ] .
\end{gather*}
Here we used that \eqref{eq:G homology cycle} and \eqref{eq:lat F2} imply $\sn_j\cdot A=\nul$. The (column) rank of both matrices $\EV_\cF$ and $\EV_\V$ is $|\ospv|+1$. So there is an invertible matrix $G$ with entries in $\QQ$ such that $\EV_\V=\EV_\cF\cdot G$. Putting things together we have for $j=1,2$:
\begin{gather*}
[ \gl_{\V,j} | \nul ]=[ \gl_j | \nul ]\cdot G .
\end{gather*}

\pph{}\label{pph:lifted vertices} Next we determine the vertices in the tiling of $\RR^2$ corresponding to a realization $\V$ of~$\SFL$. Fix $\bv_0\in\ospv$ and set $z_{\bv_0}=\nul\in\RR^2$. For $\bv\in\ospv$ with $\bv\neq\bv_0$ take a path $\bp_\bv$ in $\gGL$ starting at $\bv_0$ and ending at $\bv$, say $\bp_\bv=(e_1,\ldots,e_k)$, and set $z_\bv=\V(e_1)+\dots+\V(e_k)$. Then we obviously have:

\begin{Lemma}\label{lem:lifted vertices} The set of vertices in the tiling of $\RR^2$ which lie over the vertex $\bv$ of $\gGL \subset \RR^2/\gL_\V$ is precisely $z_\bv+\gL_\V$.
\end{Lemma}

\pph{}\label{sec:draw GD} For a realization $\V$ and a positive fractional matching $\theta$ we now compute for every polygon in the tiling the point which is the convex combination specified by $\theta$ of the midpoints of the edges of that polygon. In order to do this in an efficient way we fix a perfect matching~$\sm$. Consider a~black polygon $\bb\in\ospb$. Let $(e_1,\ldots,e_q)$ be the corresponding cycle of the permutation~$\gs_1$ written such that $e_1\in\sm$. Then the vertices of the polygon~$\bb$ are located at
\begin{gather*}
z_\bv+\sum_{j=1}^h\V(e_j)\qquad\textrm{for}\quad 1\leq h\leq q ,
\end{gather*}
where $\bv=s(e_1)$ and $z_\bv$ is as in Section~\ref{pph:lifted vertices}. The midpoints of its sides are
\begin{gather*}
z_\bv+\breuk{1}{2}\V(e_h)+\sum_{j=1}^{h-1}\V(e_j)\qquad\textrm{for}\quad 1\leq h\leq q .
\end{gather*}
The convex combination of these midpoints specified by $\theta$ is therefore
\begin{gather}\label{eq:marked point}
B_{\bb,\sm,\theta,\V}=z_\bv+\sum_{h=1}^q\theta(e_h)\left(\breuk{1}{2}\V(e_h)+\sum_{j=1}^{h-1}\V(e_j)\right).
\end{gather}
For a white polygon $\bw$ one can construct in the same way a point $W_{\bw,\sm,\theta,\V}$.

Using translations from the lattice $\gL_\V$ one subsequently obtains a marked point in every polygon in the tiling of $\RR^2$ and vectors connecting this point to the vertices of the polygon.

\medskip

\pph{}\label{pph:quad} Using the map $\gb_\bb\colon \ocE\rightarrow\ZZ$ from \eqref{eq:black cycle} and the matrix $\rho_{\sm,1}$ from~\eqref{eq:broken} we can rewrite~\eqref{eq:marked point} as
\begin{gather}\label{eq:marked point 2}
B_{\bb,\sm,\theta,\V} = z_\bv+\theta^t\cdot\diag(\gb_\bb)\cdot\big({-}\breuk{1}{2}\II+\rho_{\sm,1}\big)\cdot\V ,
\end{gather}
where on the right-hand side we view $\theta^t$ as a row vector and $\V$ as a $|\ocE|\times 2$-matrix with real entries. The source point of edge $e_k$ of $\bb$ is located at
\begin{gather}\label{eq:source k}
 z_\bv+\sum_{j=1}^{k-1}\V(e_j) = z_\bv+\big(\textrm{$e_k$-th row of matrix} \ (-\II+\rho_{\sm,1})\cdot\V\big) .
\end{gather}

The vector from the source point of edge $e_k$ of $\bb$ to the marked point $B_{\bb,\sm,\theta,\V}$ in $\bb$ is obtained by subtracting~\eqref{eq:source k} from \eqref{eq:marked point 2}. The term $z_\bv$ cancels out. Noticing that $\bb=b(e_k)$ we are led to introduce the $|\ocE|\times |\ocE|$-matrix $\mathbf{B}_{\sm,\theta}$ by
\begin{gather}\label{eq:bigB}
\text{$e$-th row of } \mathbf{B}_{\sm,\theta} = \theta^t\cdot\diag(\gb_{b(e)})\cdot\big({-}\breuk{1}{2}\II+\rho_{\sm,1}\big) .
\end{gather}
The vector from the source point of edge $e$ to the marked point in the polygon $b(e)$ is then the $e$-th row of the matrix $(\mathbf{B}_{\sm,\theta}+\II-\rho_{\sm,1})\cdot\V$.

Proceeding in the same way for the white polygons we define the $|\ocE|\times |\ocE|$-matrix $\mathbf{W}_{\sm,\theta}$ by
\begin{gather}\label{eq:bigW}
\text{$e$-th row of } \mathbf{W}_{\sm,\theta} = \theta^t\cdot\diag(\gb_{w(e)})\cdot\big({-}\breuk{1}{2}\II+\rho_{\sm,0}\big) .
\end{gather}

The following proposition summarizes our findings about the quadrangles.

\begin{Proposition}\label{prop:quadrangle} In the quadrangle corresponding to $e\in\ocE$
$$
\includegraphics{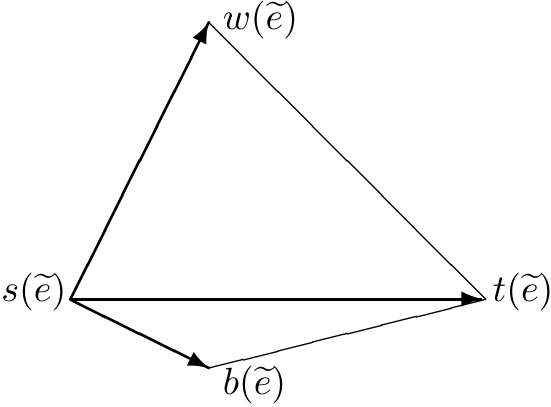}
$$
\begin{enumerate}\itemsep=0pt{\samepage
\item[$(i)$] the vector from $s(\widetilde{e})$ to $b(\widetilde{e})$ is the $e$-th row of the $|\ocE|\times 2$-matrix
\begin{gather*}\bigl(\mathbf{B}_{\sm,\theta}+\II-\rho_{\sm,1}\bigr)\cdot\V ,\end{gather*}
\item[$(ii)$] the vector from $s(\widetilde{e})$ to $w(\widetilde{e})$ is the $e$-th row of the $|\ocE|\times 2$-matrix
\begin{gather*}\bigl(\mathbf{W}_{\sm,\theta}+\II-\rho_{\sm,0}\bigr)\cdot\V ,\end{gather*}
\item[$(iii)$] the vector from $s(\widetilde{e})$ to $t(\widetilde{e})$ is the $e$-th row of the $|\ocE|\times 2$-matrix $\V$,}
\end{enumerate}
In the application of this result to the weight realization of $\SFL$ given by the triple of positive weight functions $(\nu_1,\nu_2,\nu_3)$ $($see Definition~{\rm \ref{def:weight realization})} one takes $\V=(\nu_1-\nu_3, \nu_2-\nu_3)$ and $\theta=\frac{1}{\deg \nu_3}\nu_3$.
\end{Proposition}

\pph{} By combining Sections~\ref{pph:latweight}, \ref{pph:lifted vertices}, \ref{sec:draw GD} and \ref{pph:quad} we can draw the quadrangle tiling as follows. For $e\in\ocE$ take $z_{s(e)}$ as in as in Section~\ref{pph:lifted vertices} and draw for every $z\in z_{s(e)}+\gL_\V$ the quadrangle as in Proposition~\ref{prop:quadrangle} with vertex $s(\widetilde{e})$ located at $z$ and with the vectors from $s(\widetilde{e})$ to~$b(\widetilde{e})$, $w(\widetilde{e})$ and~$t(\widetilde{e})$ as specified in Proposition~\ref{prop:quadrangle}.

\LastPageEnding

\end{document}